\documentclass[12pt,reqno]{amsart}
\usepackage{amsmath, amsthm, amscd, amssymb, amsfonts, latexsym}
\usepackage{cases}
\usepackage{fullpage}
\usepackage{bm}
\usepackage[all]{xy}
\usepackage{tikz}
\usepackage{mathrsfs}

\newcommand{\kommentar}[1]{}
\newcommand{\tf}{\frac{a(d-1/2)}{r}}

\newcommand{\acom}[1]{{\color{blue}{Alexandra: #1}} }

\newcommand{\rat}{\frac{a(d-1/2)}{r}}

\newcommand{\avg}{\frac{1}{|\mathcal{H}_{2g+1}|}}

\newcommand{\sumstar}{\sideset{}{^*}\sum}

\renewcommand{\pmod}[1]{\,(\mathrm{mod}\,#1)}

\newtheorem{lem}{Lemma}[section]

\newtheorem{thm}[lem]{Theorem}

\newtheorem{cor}[lem]{Corollary}
\newtheorem{conj}[lem]{Conjecture}
\theoremstyle{definition}

\begin{document}

\author{Alexandra Florea}
\address{UC Irvine, Mathematics Department, Rowland Hall, Irvine 92697, USA}
\email{floreaa@uci.edu}

\title{Negative moments of $L$--functions with small shifts over function fields}

\begin{abstract}
We consider negative moments of quadratic Dirichlet $L$--functions over function fields. Summing over monic square-free polynomials of degree $2g+1$ in $\mathbb{F}_q[x]$, we obtain an asymptotic formula for the $k^{\text{th}}$ shifted negative moment of $L(1/2+\beta,\chi_D)$, in certain ranges of $\beta$ (for example, when roughly $\beta \gg   \log g/g $ and $k<1$). We also obtain non-trivial upper bounds for the $k^{\text{th}}$ shifted negative moment when $\log(1/\beta) \ll \log g$. Previously, almost sharp upper bounds were obtained in \cite{ratios} in the range $\beta \gg g^{-\frac{1}{2k}+\epsilon}$.

\end{abstract}

%\subjclass[2010]{11M06, 11M38, 11R16, 11R58}
%\keywords{Moments over function fields, cubic twists, non-vanishing.}

\maketitle

\section{Introduction}

Let $M_k(T)$ denote the $2k^{\text{th}}$ moment of the Riemann zeta-function. Namely, we let
$$ M_k(T) = \int_0^T \left| \zeta ( \tfrac{1}{2} +it )\right|^{2k}  dt.$$
Hardy and Littlewood \cite{hl} showed that $M_1(T) \sim T \log T$, and Ingham \cite{ingham} showed that $M_2(T) \sim \frac{1}{2 \pi^2} T (\log T)^4$. It is conjectured that 
$$M_k(T) \sim A_k T (\log T)^{k^2},$$ for some constant $A_k$, whose precise value was predicted by Keating and Snaith \cite{ksnaith}, using analogies with random matrix theory. No moment higher than $4$ has been rigorously computed so far. Soundararajan \cite{sound_ub} obtained almost sharp upper bounds, conditional on the Riemann hypothesis. More precisely, he showed that $M_k(T) \ll T (\log T)^{k^2+\epsilon},$ for any $\epsilon>0$. Refining Soundararajan's method, Harper \cite{harper} obtained upper bounds of the correct order of magnitude for moments of the Riemann zeta-function, by removing the $\epsilon$ on the power of $\log T$.

Focusing on the family of quadratic Dirichlet $L$--functions, Jutila \cite{jutila} obtained asymptotics for the first and second moment of this family. He showed that
 $$\sumstar_{0 < d \leq D} L \Big(\tfrac{1}{2}, \chi_d \Big) \sim C_1 D \log D,$$  where the sum above is over fundamental discriminants, and that
 $$ \sumstar_{0 < d \leq D} L \Big(\tfrac{1}{2}, \chi_d \Big)^2 \sim C_2 D (\log D)^{3},$$  for some explicit constants $C_1$ and $C_2$.
Soundararajan \cite{sound} obtained an asymptotic formula for the second moment with a power savings error term, and also obtained an asymptotic for the third moment. The cubic moment was independently computed using multiple Dirichlet series in \cite{dgh}. More recently, a lower order term of size $D^{3/4}$ was explicitly computed for the cubic moment by Diaconu and Whitehead  \cite{dw} and by Diaconu in the function field setting \cite{diaconu}. Conditional on the Generalized Riemann Hypothesis, Shen \cite{shen} obtained an asymptotic with the leading order term for the fourth moment. 
Generally, it is conjectured that
 \begin{equation}
 \sumstar_{0 < d \leq D} L \Big(\tfrac{1}{2}, \chi_d \Big)^k \sim C_k D (\log D)^{\frac{k(k+1)}{2}}, \label{momk}
 \end{equation}  and the precise value of $C_k$ follows from work of Keating and Snaith \cite{ks2}, again using random matrix theory. Conrey, Farmer, Keating, Rubinstein and Snaith \cite{cfkrs} further refined this conjecture, including lower order terms in the asymptotic formula \eqref{momk}.
 The approach used by Soundararajan and Harper in obtaining upper bounds for moments of $\zeta(s)$ yields upper bounds of the right order of magnitude for the family of quadratic Dirichlet $L$--functions conditional on GRH, while work of Soundararajan and Rudnick \cite{sound_rudnick} provides unconditional lower bounds of the right order of magnitude.  Over function fields, the first moment was computed by Andrade and Keating \cite{ak}, and a lower order term identified in \cite{F1}. Higher moments, up to the fourth, were obtained in \cite{F2, F4}, as well as almost sharp upper bounds on all the positive moments.
 
 While all the results mentioned above hold for positive moments in families of $L$--functions, much less is known about negative moments, even at a conjectural level. In the case of the Riemann zeta-function $\zeta(s)$, a conjecture due to Gonek \cite{gonek} states the following. 
\begin{conj}[Gonek] \label{gonek_conj}
Let $k>0$ be fixed. Uniformly for $1 \leq \delta \leq \log T$,
$$ \frac{1}{T}\int_1^T \Big| \zeta \Big(\frac{1}{2}+ \frac{\delta}{\log T}+it \Big) \Big|^{-2k} \, dt \asymp  \Big(\frac{\log T}{\delta} \Big)^{k^2}, $$
and uniformly for $0 < \delta \leq 1$,
$$
\frac{1}{T}\int_1^T \Big| \zeta \Big(\frac{1}{2}+ \frac{\delta}{\log T}+it \Big) \Big|^{-2k} \, dt \asymp 
\begin{cases}
(\log T)^{k^2} & \mbox{ if } k<1/2, \\
(\log\frac{e}{\delta}) (\log T)^{k^2} & \mbox{ if } k=1/2, \\
\delta^{1-2k} (\log T)^{k^2} & \mbox{ if } k>1/2.
\end{cases}
$$
\end{conj}
Random matrix theory inspired ideas (see \cite{kb, kf}) seem to suggest certain transition regimes in the formulas above when $k = (2n+1)/2$, for $n$ a positive integer. 
While obtaining lower bounds for the negative moments is a more tractable problem (Gonek \cite{gonek} proved lower bounds of the conjectural correct order of magnitude for $1 \leq \delta \leq \log T$ and all $k>0$ and for $0<\delta\leq 1$ for $k<1/2$ conditional on the Riemann Hypothesis), obtaining upper bounds is a more difficult problem, and no progress has been made so far on the problem in any family of $L$--functions (recent work in progress of the author and H. Bui addresses the question of obtaining upper bounds in some ranges of $\delta$).

In the case of quadratic Dirichlet $L$--functions, when studying the $k^{\text{th}}$ negative moment, random matrix theory computations due to Forrester and Keating \cite{kf} seem to suggest certain transition regimes for small shifts (i.e. shifts smaller than $1/\log X$, where $X$ is roughly the size of the family.) More precisely, the computations in \cite{kf} suggest certain jumps in the asymptotic formulas when $k=2j+1/2$, and $j$ a positive integer. 

Very recently, almost sharp upper bounds were obtained for negative moments of quadratic Dirichlet $L$--functions over function fields \cite{ratios} when the shift in the $L$--function is big enough. Specifically, if $\mathcal{H}_{2g+1}$ denotes the ensemble of monic, square-free polynomials of degree $2g+1$ over $\mathbb{F}_q[x]$, let $L(s,\chi_D)$ denote the $L$--function associated to the quadratic character $\chi_D$. Then it is shown in \cite{ratios} that for $\beta \gg g^{-\frac{1}{2k}+\epsilon}$, we have
\begin{equation}
\label{previous}
 \frac{1}{|\mathcal{H}_{2g+1}|} \sum_{D \in \mathcal{H}_{2g+1}} \frac{1}{|L(1/2+\beta,\chi_D)|^k} \ll  \Big( \frac{1}{\beta} \Big)^{\frac{k(k-1)}{2}} (\log g)^{\frac{k(k+1)}{2}}.
 \end{equation}
Note that it is expected that the upper bound above is sharp, up to the logarithmic factor.

In this paper, we treat the range when $\beta \ll g^{-\frac{1}{2k}+\epsilon}$, which is more difficult. The closer we are to the critical line, the more difficult the problem becomes, due to the closer proximity of zeros. %It is not clear what the true size of the negative moments with small shifts is, in light of the previous discussion of Gonek's conjecture for $\zeta(s)$ and the random matrix theory analogies. 
Here, we obtain non-trivial upper bounds for small shifts $\beta$ with $\log(1/\beta)\ll \log g$. In certain ranges where $\beta$ is big enough (i.e., $\beta \gg g^{-1/k+\epsilon}$), we prove a more precise analogue of Gonek's conjecture, obtaining an asymptotic formula. We remark that asymptotic formulas for negative moments of $L(1,\chi_d)$ were obtained by Granville and Soundararajan \cite{gs} in the number field setting and by Lumley \cite{lumley} in the function field setting. The techniques used in those papers are different, as one considers moments far from the critical point $1/2$, and the $L$--functions in those cases can be modeled by random Euler products.  In our work, we obtain asymptotic formulas or upper bounds when the shift goes to zero with the size of the family. %As far as we are aware, our work provides the first upper bounds in the literature for small shifts. 
More precisely, we prove the following.

\begin{thm}
\label{theorem_ub}
Let $k>0$ and $\beta>0$ such that $\log(1/\beta) \ll \log g$. 
Then
\begin{equation*}
\frac{1}{|\mathcal{H}_{2g+1}|} \sum_{D \in \mathcal{H}_{2g+1}} \frac{1}{ \Big|L \Big(\frac{1}{2}+\beta+it,\chi_D \Big) \Big|^k}
\end{equation*}
\begin{numcases}{ \ll }
(\log g)^{k(k+2)/2}g^{k(k+1)/2} &  if  $\beta \gg g^{-\frac{1}{k}+\epsilon}, \, k \geq 3/2$, \label{improved}  \\
(\log g)^{7k/4} g^{5k/4} &  if  $\beta \gg g^{-\frac{1}{k}+\epsilon}, \, k<3/2,$ \label{second_bd} \\
q^{g \big(k \frac{\log(1/\beta)}{\log g}-1+\epsilon \big)} &  if $ \beta=o(g^{-\frac{1}{k}+\epsilon})$.\label{third_bd}
\end{numcases}
%\end{equation*}
\end{thm}

Note that Theorem \ref{theorem_ub} above holds for any $t$ (the $L$--function is periodic as a function of $t$).

We also refine Theorem \ref{theorem_ub} to obtain an asymptotic formula in the following case.

\begin{thm}
\label{asymptotic}
Let $k>0$ and $\epsilon>0$. Then for $\Re \beta \gg \max \{ g^{-\frac{1}{k}+\epsilon}, \log g/g\}$, we have 
\begin{align*}
\avg \sum_{D \in \mathcal{H}_{2g+1}} & \frac{1}{ L \Big( \frac{1}{2}+\beta,\chi_D \Big)^k } = \zeta_q(1+2\beta)^{\binom{k}{2}} A(1;\beta) +O \Big(q^{-g \Re (\beta)(1-\epsilon)} g^{1+ \frac{k}{2} \Big(1+ \max \{k,3/2\} \Big)} \\
& \times (\log g)^{\frac{k}{2} (2+\max \{k,3/2\})} \Big),
\end{align*}
with $A(1;\beta)$ given in equation \eqref{ab}.
\end{thm}

We note that the theorem above provides an asymptotic formula when $k \geq 1$ and $\Re \beta \gg g^{-\frac{1}{k}+\epsilon}$. If $k<1$, then one needs $\Re \beta \geq c_k \log g/ g$, for $c_k$ a specific constant depending on $k$.  We record this in the following corollary.

\begin{cor}
\label{corol}
Let $k>0$, $\epsilon>0$ and $C>0$. Then for $\Re \beta \geq \max \{ Cg^{-\frac{1}{k}+\epsilon}, (1+\epsilon) \frac{(\frac{7k}{4}-\frac{k^2}{2}+1) \log_q g}{g}$, we have
$$ \avg \sum_{D \in \mathcal{H}_{2g+1}} \frac{1}{ L \Big( \frac{1}{2}+\beta,\chi_D \Big)^k } = \zeta_q(1+2\beta)^{\binom{k}{2}} A(1;\beta)+o(1).$$
\end{cor}
Note that Corollary \ref{corol} allows one to obtain an asymptotic formula for the negative moments when $\beta$ is as small as roughly $ \log g/g$, as long as $k<1$. We note that the term $1+7k/4-k^2/2$ could be slightly improved in the corollary above, but we have decided not to focus on that. It would be of interest to be able to obtain asymptotic formulas in the range $\beta \gg \log g/g$ for all values of $k$.

The organization of the paper is as follows. In section \ref{background} we provide some background and the preliminary lemmas we will use throughout the paper. We prove Theorem \ref{theorem_ub} in section \ref{proof_mainthm} and Theorem \ref{asymptotic} in section \ref{section_asymp}. The proof of Theorem \ref{theorem_ub} starts in a similar way as the proof of Theorem $1.3$ in \cite{ratios}, and uses sieve theoretic inspired ideas. This circle of ideas has  recently been used successfully in a variety of settings, as in \cite{sound_ub, harper, maks_sound, LR}. 

The difference from \cite{ratios} which allows one to double the range of $\beta$ (in Theorem \ref{theorem_ub} one obtains almost sharp bounds for $\beta \gg g^{-1/k+\epsilon}$ as opposed to $\beta \gg g^{-1/(2k)+\epsilon}$ in \cite{ratios}) is the use of the large sieve for quadratic characters rather than simple orthogonality of characters. However, the quadratic large sieve introduces a factor of $q^{\epsilon g}$ in the upper bound, hence one needs to use more care to refine the initial  bound of size $q^{\epsilon g}$ to a bound of the form $g^{O(1)}$. When performing the first step of the argument, one has to use an priori bound for negative moments coming from a pointwise bound for the inverse $L$--function. Once we obtain the upper bound of size $q^{\epsilon g}$ (and keep track on the dependence on $k$ in the bound), we do the second step of the argument, but use as an a priori bound for the negative moments the bound obtained in the previous step. When the shift is bigger than $g^{-1/k+\epsilon}$, the argument described above gives an almost sharp upper bound, up to some logarithmic factors. This allows us to further refine the result and obtain the asymptotic formula in Theorem \ref{asymptotic} in that range. 
\newline 
\newline
%Finally, we remark that similar ideas can be used to approach Gonek's conjecture \ref{gonek_conj} for the Riemann zeta-function.

\kommentar{
However, the approach in \cite{ratios} alone only yields a weak upper bound for the negative moments. The new ingredient in the proof of Theorem \ref{theorem_ub} is the use of a more delicate, inductive argument to obtain stronger upper bounds. The first step of the inductive argument is similar to the work in \cite{ratios}, giving a small saving over the trivial bound. One ingredient needed in the proof is the use of an a priori bound for the size of the negative moments. The first step of the inductive argument uses the a priori bound coming from pointwise bounds for the $L$--functions. At each step of the inductive argument, we use the a priori bound obtained in the previous step of the induction. When the shift is bigger than $(\log g)/g$, the argument described above gives an almost sharp upper bound for the negative moments, up to some logarithmic factors.  \acom{rewrite}}
%\newline
%\newline
 \textsc{Acknowledgments.} The author thanks Z. Rudnick for a helpful comment on a previous version of this paper and gratefully acknowledges support from NSF grant DMS-2101769 while working on this.

\section{Background in function fields}
\label{background}

Here we gather some basic facts about $L$--functions in function fields. Many of the proofs can be found in \cite{rosen}.

%Throughout the paper, for simplicity, we will take $q$ to be a prime with $q \equiv 1 \pmod 4$. 
Let $\mathcal{M}$ denote the set of monic polynomials over $\mathbb{F}_q[x]$, $\mathcal{M}_n$ the set of monic polynomials of degree $n$, $\mathcal{M}_{\leq n}$ the set of monic polynomials of degree at most $n$ and $\mathcal{M}_{\geq n}$ the set of monic polynomials with degree at least $n$. Let $\mathcal{H}_n$ denote the set of monic, square-free polynomials of degree $n$, and $\mathcal{P}$ the ensemble of monic, irreducible polynomials. The symbol $P$ will stand for a monic, irreducible polynomial.
Note that
$|\mathcal{M}_n|=q^n$, and for $n \geq 1$, $| \mathcal{H}_n| = q^{n-1}(q-1)$.

For a polynomial $f$ in $\mathbb{F}_q[x]$, let $|f|:=q^{\deg(f)}$ denote the norm of $f$. For $\Re(s)>1$, the zeta-function of $\mathbb{F}_q[x]$ is defined by
$$ \zeta_q(s) = \sum_{f \text{ monic}} \frac{1}{|f|^s}= \prod_P (1-|P|^{-s})^{-1}.$$

Since $|\mathcal{M}_n|=q^n$, we see that
\[
\zeta_q(s)=\frac{1}{1-q^{1-s}}.
\]
It is sometimes convenient to make the change of variable $u=q^{-s}$, and then write $\mathcal{Z}(u)=\zeta_q(s)$, so that $$\mathcal{Z}(u)=\frac{1}{1-qu}.$$
The M\"{o}bius function $\mu$ is defined as usual by $\mu(f)=(-1)^{\omega(f)}$ if $f$ is a square-free polynomial and where $\omega(f) =\sum_{P|f} 1$, and $0$ otherwise.

The Prime Polynomial Theorem states that 
\begin{equation}
 \sum_{\substack{P \in \mathcal{P} \\ \deg(P) =n}} 1 = \frac{q^n}{n} + O \Big( \frac{q^{n/2}}{n} \Big).
 \label{ppt}
\end{equation}
We will also use the Prime Polynomial Theorem in the less precise form
\begin{equation}
\sum_{\substack{P \in \mathcal{P} \\ \deg(P) =n}} 1  \leq \frac{q^n}{n}.
\label{ppt_ineq}
\end{equation}
(See, for example, formula $2.1$ in \cite{DFL}.)

For $P$ a monic irreducible polynomial, the quadratic residue symbol $\big(\frac{f}{P}\big)\in\{0,\pm1\}$ is defined by
\[
\Big(\frac{f}{P}\Big)\equiv f^{(|P|-1)/2}(\textrm{mod}\ P).
\]
If $Q=P_{1}^{\alpha_1}P_{2}^{\alpha_2}\ldots P_{r}^{\alpha_r}$, then the Jacobi symbol is defined by
\[
\Big(\frac{f}{Q}\Big)=\prod_{j=1}^{r}\Big(\frac{f}{P_j}\Big)^{\alpha_j}.
\]
The Jacobi symbol satisfies the quadratic reciprocity law. Namely, if $A,B\in \mathbb{F}_q[x]$ are relatively prime, monic polynomials, then
\[
\Big(\frac{A}{B}\Big)=(-1)^{(q-1)\deg(A)\deg(B)/2}\Big(\frac{B}{A}\Big).
\]
%Since $q\equiv 1(\textrm{mod}\ 4)$,  the quadratic reciprocity law gives $\big(\frac{A}{B}\big)=\big(\frac{B}{A}\big)$.

For $D$ monic, we define the character 
\[
\chi_D(g)=\Big(\frac{D}{g}\Big),
\]
and consider the $L$-function attached to $\chi_D$,
\[
L(s,\chi_D):=\sum_{f\in\mathcal{M}}\frac{\chi_D(f)}{|f|^s}.
\]
With the change of variable $u=q^{-s}$ we have
\begin{equation}\label{formulaL}
\mathcal{L}(u,\chi_D):=L(s,\chi_D)=\sum_{f\in\mathcal{M}}\chi_D(f)u^{d(f)}=\prod_{P\in \mathcal{P}}\big(1-\chi_D(P)u^{d(P)}\big)^{-1}.
\end{equation}
For $D\in\mathcal{H}_{2g+1}$, $\mathcal{L}(u,\chi_D)$ is a polynomial in $u$ of degree $2g$ satisfying the functional equation
\begin{equation}\label{tfe}
\mathcal{L}(u,\chi_D)=(qu^2)^g\mathcal{L}\Big(\frac{1}{qu},\chi_D\Big).
\end{equation}

 The Riemann Hypothesis for curves over function fields was proven by Weil \cite{weil}, so all the zeros of $\mathcal{L}(u,\chi_D)$ are on the circle $|u|=q^{-1/2}$.

 We will use the following pointwise upper bound for the inverse of the $L$--function.
 \begin{lem} \label{pointwise}
 For $0<\beta \ll \frac{1}{\log g}$ and $t \in \mathbb{R}$, we have  
 $$  \frac{1}{|L(1/2+\beta+it,\chi_D)|} \leq \exp \Bigg( \frac{(1+\epsilon)g}{\log_q g} \log \Big( \frac{1}{\beta} \Big) \Bigg).$$
 \end{lem}
\begin{proof}
See Lemma $5.3$ and Remark $5.1$ in \cite{ratios}.
\end{proof}

We also need the following estimates.

\begin{lem}\label{count}
For $f\in\mathcal{M}$ we have
\[
 \sum_{D \in \mathcal{H}_{2g+1}} \chi_D(f^2)=|\mathcal{H}_{2g+1}|  \prod_{P|f}\bigg(1+\frac{1}{|P|}\bigg)^{-1}+O_\varepsilon(  |f|^{\varepsilon}).
\]
\end{lem}
\begin{proof}
See, for example, Lemma 3.4 in \cite{hybrid}.
\end{proof}

\begin{lem}
\label{notpv}
For $f$ not a square polynomial, we have
$$ \Big| \sum_{D \in \mathcal{H}_{2g+1}} \chi_D(f)  \Big|\ll q^{g} |f|^{\epsilon}.$$
\end{lem}
\begin{proof}
See Lemma $3.5$ in \cite{hybrid}.
\end{proof}
\kommentar{We also need the following estimates.

\begin{lem}\label{count}
For $f\in\mathcal{M}$ we have
\[
\avg \sum_{D \in \mathcal{H}_{2g+1}} \chi_D(f^2)= \prod_{P|f}\bigg(1+\frac{1}{|P|}\bigg)^{-1}+O_\varepsilon( q^{-2g} |f|^{\varepsilon}).
\]
\end{lem}
\begin{proof}
See, for example, Lemma 3.4 in \cite{hybrid}.
\end{proof}

\begin{lem}
\label{notpv}
For $f$ not a square polynomial, we have
$$ \Big|\avg  \sum_{D \in \mathcal{H}_{2g+1}} \chi_D(f)  \Big|\ll q^{-g} |f|^{\epsilon}.$$
\end{lem}
\begin{proof}
See Lemma $3.5$ in \cite{hybrid}.
\end{proof}}
Throughout the paper, we will frequently use the following analogue of Perron's formula in function fields. If the power series $\sum_{n=0}^{\infty} a(n) u^n$ is absolutely convergent in $|z| \leq r <1$, then
\begin{equation}
 \sum_{n \leq N} a(n) = \frac{1}{2 \pi i} \oint_{|u|=r} \Big( \sum_{n=0}^{\infty} a(n)u^n) \Big) \frac{du}{(1-u)u^{N+1}}. 
 \label{perron}
\end{equation}

Now let $t \in \mathbb{R}$ and $\ell$ be an even integer. Let
$$E_{\ell}(t) = \sum_{s \leq \ell} \frac{t^s}{s!}.$$
Note that we have $E_{\ell}(t) >0$ for any $t$ since $\ell$ is even. We will  use the fact that  for $t \leq \ell/e^2$, we have 
\begin{equation}
e^t \leq (1+e^{-\ell/2}) E_{\ell}(t).
\label{taylor}
\end{equation}
For a proof, see, for example, \cite{maks_sound}.

Let $\nu(f)$ be the multiplicative function given by $$\nu(P^a) = \frac{1}{a!}.$$ Let $\Omega(f)$ denote the number of prime factors of $f$, counting multiplicity. We will use the following result (see Lemma $3.2$ in \cite{DFL}).

\begin{lem}\label{power}
Let $a(f)$ be a completely multiplicative function% from $\mathbb{F}_q[t]$ to $\mathbb{C}$
. Then for any interval $I$ and any $s\in\mathbb{N}$ we have that
\begin{equation*}
\bigg( \sum_{\deg(P) \in I} a(P)  \bigg)^s = s! \sum_{\substack{P | f \Rightarrow \deg(P) \in I \\ \Omega(f) = s}} a(f) \nu(f).
\end{equation*}
\end{lem}
We will also need the following form of the quadratic large sieve over function fields \cite{f_s}.
\begin{lem}
\label{large_sieve}
Let $a(f)$ be arbitrary complex numbers supported on monic polynomials, and let $n=O(g)$. We have 
$$\sum_{D \in \mathcal{H}_{2g+1}} \Big| \sum_{f \in \mathcal{M}_n} a(f) \chi_D(f)\Big|^2 \leq (q^{2g}+q^n) q^{A g/(\log g)^{1/4}} \sum_{f_1f_2= \square} |a(f_1)a(f_2)|,$$ for some absolute constant $A>0$.
\end{lem}

We note that the result in \cite{f_s} could be improved to obtain a better bound than $q^{Ag/(\log g)^{1/4}}$, but for the purpose of this paper, Lemma \ref{large_sieve} above is enough.
 
\section{Setup of the proof and initial lemmas}
We will first introduce some of the ideas in the proof of Theorem \ref{theorem_ub}, and will state some key lemmas. We will return to the proof of Theorem \ref{theorem_ub} in section \ref{proof_mainthm}.

%To prove Theorem \ref{theorem_ub}, we will use an inductive approach. At step $m$, we divide the primes into $K_m+1$ intervals depending on their degree, and denote the intervals by $I_{m,j}$ for $0 \leq j \leq K_m$. 
Let 
$$I_{0} = (0,N_{0}] , I_{1}=(N_{0},N_{1}], \ldots, I_{K} = (N_{K-1},N_{K}],$$ where $N_{j}$ are parameters we will choose later. Also let $s_{j}$ and $\ell_{j}$ be even integers which we will choose later on. For now, we can think of $s_j N_j \asymp g$ and $\sum_{h=0}^K \ell_h N_h \ll g$. %for $0 \leq j \leq K_m$, such that for $1 \leq j \leq K_m$, we have
 Let 
 \begin{align*}
 a_{\beta}(P;N) &=-\cos(t \deg(P) \log q)\sum_{j=0}^{\infty}   \Big(\frac{(j+1)\deg(P) q^{-j(N+1)\beta}}{\deg(P)+j(N+1)} \\
 &- \frac{(j+1) \deg(P) q^{-(j+2)(N+1)\beta}|P|^{2\beta}}{(j+2)(N+1)-\deg(P)} \Big) . 
 \end{align*}
 We extend $a_{\beta}(P;N)$ to a completely multiplicative function in the first variable.
As in \cite{ratios}, for $\deg(P) \leq N$, we have
\begin{equation}
| a_{\beta} (P;N)| \leq 1+ \frac{1}{q^{(N+1)\beta}-1},
\label{ineq1}
\end{equation} if $N \beta \gg 1$, and 
\begin{equation}
|a_{\beta}(P;N)| \leq \Big( \frac{1}{2}+\epsilon \Big)\Big(  \log  \frac{1}{N\beta} \Big)
\label{nb}
\end{equation} if $N \beta =o( 1)$. Note that in \cite{ratios} we used the weaker bound $|a_{\beta}(P;N) \ll \log \Big( \frac{1}{\beta} \Big)$ which was enough for our purposes, but the stronger bound above easily follows from \cite{ratios}; see the equation before $(5.20)$.

We rewrite \eqref{ineq1} and \eqref{nb} into a single inequality as  
\begin{equation}
|a_{\beta}(P;N)| \leq B(N) \Big(\log \frac{1}{N \beta} \Big)^{\gamma(N)},\label{single}
\end{equation}
where $\gamma(N)=1$ if $N \beta=o(1)$ and $\gamma(N)=0$ if $N \beta \gg 1$, and where
\begin{equation}
\label{b*}
B(N) = 
\begin{cases}
\frac{1}{2}+\epsilon & \mbox{ if } N \beta=o(1) \\
1+\frac{1}{q^{(N+1)\beta)}-1} & \mbox{ if } N \beta \gg 1.
\end{cases}
\end{equation}
It then follows that
\begin{equation}
\label{combined}
|a_{\beta}(f;N) \leq B(N)^{\Omega(f)}  \Big(\log \frac{1}{N \beta} \Big)^{\gamma(N) \Omega(f)}.
\end{equation}
 For $0 \leq  j,h \leq K$, let
 $$P_{I_{h}}(D;N_{j}) = \sum_{\deg(P) \in I_{h}} \frac{a_{\beta}(P;N_{j}) \chi_D(P)}{|P|^{1/2+\beta}}.$$
 
 Similarly as in \cite{ratios}, for $h \leq K$,  let 
$$\mathcal{T}_{h} =  \Big\{ D \in \mathcal{H}_{2g+1} \, | \,   \max_{h \leq u \leq K} \big| P_{I_{h}} ( D; N_{u}) \big|  \leq \frac{\ell_{h}}{ke^2} \Big\}.$$
A minor modification of Lemma $5.4$ in \cite{ratios} gives the following lemma:
\begin{lem}
\label{lem_initial}
We either have
$$ \max_{0 \leq u \leq K}  \Big|  P_{I_{0}} (D; N_{u}) \Big| > \frac{\ell_{0}}{ke^2},$$ or 
\begin{align}
\label{ineq211}
 \frac{1}{  | L(1/2+\beta+it,\chi_D)|^k} \leq \exp(O(k)) (S_{1}(D)+S_{2}(D)),
\end{align}
where 
$$S_{1}(D) =  \Big(1-q^{-(N_{K}+1)\beta} \Big)^{-\frac{2gk}{N_{K}+1}}  (N_{K} \log g)^{k/2} \prod_{h=0}^{K} (1+e^{-\ell_{h}/2}) E_{\ell_{h}} \Big(k P_{I_{h}}(D;N_{K})\Big),$$ and
\begin{align*}
S_{2}(D) = (\log g)^{k/2} & \sum_{0 \leq j \leq K-1} \sum_{j < u \leq K} \Big(1-q^{-(N_{j}+1)\beta} \Big)^{-\frac{2gk}{N_{j}+1}} N_{j}^{k/2}\\
& \times \prod_{h=0}^j (1+e^{-\ell_{h}/2}) E_{\ell_{h}} \Big( k P_{I_{h}} (D;N_{j})\Big)  \Big(\frac{ke^2}{\ell_{j+1}} P_{I_{j+1}} (D;N_{u}) \Big)^{s_{j+1}}.
\end{align*}

\end{lem}
\begin{proof}
To obtain \eqref{ineq211}, note that if $D \notin \mathcal{T}_{0}$, then either $D \in \mathcal{T}_{h}$ for all $h \leq K$ or there exists some $0 \leq j \leq K-1$ such that $D \in \mathcal{T}_{h}$ for all $h \leq j$, but $D \notin \mathcal{T}_{j+1}$. If $D \in \mathcal{T}_{h}$ for all $h \leq K$, then following the proof of Lemma $5.4$ in \cite{ratios}, we have 
 \begin{align*}
&  k \log |L(\tfrac12+\beta+it,\chi_D)|  \geq \frac{2gk}{N_{K}+1} \log \Big(1-q^{-(N_{K}+1)\beta} \Big) \\
&\qquad\qquad- k\sum_{\deg(P) \leq N_{K}} \frac{  a_{\beta}(P;N_{K}) \chi_D(P)}{|P|^{1/2+\beta}} + \frac{k}{2}  \sum_{\substack{\deg(P) \leq N_{K}/2 \\ P \nmid D}} \frac{ \cos(2t  \deg(P) \log q)}{|P|^{1+2\beta}} +O(1) .
\end{align*} 
Now we use the fact that 
$$\sum_{P|D} \frac{1}{|P|} \leq \log\log g+O(1),$$ bound $\cos(2t\deg(P) \log q) \geq -1$ and use the Prime Polynomial Theorem \eqref{ppt} to get that
\begin{align*}
   k \log |L(\tfrac12+\beta+it,\chi_D)| &  \geq \frac{2gk}{N_{K}+1} \log \Big(1-q^{-(N_{K}+1)\beta} \Big)- k\sum_{\deg(P) \leq N_{K}} \frac{  a_{\beta}(P;N_{K}) \chi_D(P)}{|P|^{1/2+\beta}} \\
&-\frac{k}{2} \log N_{K}- \frac{k}{2} \log \log g+O(k).
\end{align*}
%where $\overline{t} = \min \{ t \bmod {2 \pi}, 2 \pi - (t \bmod {2 \pi})\} .$

%\\
%&\qquad \geq \frac{2gkm}{N_{1,K_1}+1} \log \Big(1-q^{-(N_{1,K_1}+1)\beta} \Big) - m\sum_{d(P) \leq N_{1,K_1}} \frac{ c(P;N_{1,K_1})  \chi_D(P)}{|P|^{1/2+\beta} }-\frac{km}{2} \log N_{1,K_1}+O(1).
%\end{align*}
 We exponentiate the expression above and use inequality \eqref{taylor}. Since $D \in \mathcal{T}_{h}$ for all $h \leq K$, we obtain the first term in \eqref{ineq211}. We similarly obtain the second term in \eqref{ineq211}, if  $D \in \mathcal{T}_{h}$ for all $h \leq j$, but $D \notin \mathcal{T}_{j+1}$ for some $j \leq K-1$.
\end{proof}

We will also need the following key lemmas.
\begin{lem}
\label{sum_primes}
For $N_{0} s_{0} \leq  4g$ and $s_0$ even, we have
$$\sum_{D \in \mathcal{H}_{2g+1}} (P_{I_{0}}(D; N_{u}))^{s_{0}} \leq q^{2g+A g/(\log g)^{1/4}}  (s_{0}/2)! 2^{s_{0}}  (eB(N_0))^{s_{0}}  \Big( \log \frac{1}{N_{0} \beta} \Big)^{\gamma(N_{0})s_{0}} (\log N_{0})^{s_{0}/2},$$ for $A>0$ an absolute constant.
\end{lem}
\begin{proof}
Using Lemma \ref{power}, we have
\begin{align}
 \sum_{D \in \mathcal{H}_{2g+1}}  P_{I_{0}}(D;N_{u}) ^{ s_{0}} = (s_{0}/2)!)^2  \sum_{D \in \mathcal{H}_{2g+1}} \Big( \sum_{\substack{P | f \Rightarrow \deg(P) \in I_{0} \\ \Omega(f) =  s_{0}/2}} \frac{a_{\beta}(f;N_{u}) \nu(f)\chi_D(f)  }{|f|^{1/2+\beta}} \Big)^2.\label{first_bound}
\end{align}
Using the large sieve in Lemma \ref{large_sieve}, we have 
\begin{align}
\sum_{D \in \mathcal{H}_{2g+1}} & \Big( \sum_{\substack{P | f \Rightarrow \deg(P) \in I_{0} \\ \Omega(f) =  s_{0}/2}} \frac{a_{\beta}(f;N_{u}) \nu(f)\chi_D(f)  }{|f|^{1/2+\beta}} \Big)^2 \leq (q^{2g}+q^{\frac{N_{0}s_{0}}{2}})q^{A g/(\log g)^{1/4}} \nonumber \\
& \times \sum_{\substack{P|f_1 f_2 \Rightarrow \deg(P) \in I_{0} \\ \Omega(f_1) = \Omega(f_2) = s_{0}/2 \\ f_1 f_2= \square}} \frac{|a_{\beta}(f_1;N_{u}) a_{\beta}(f_2;N_{u})| \nu(f_1) \nu(f_2)}{|f_1f_2|^{1/2+\beta}}. \label{to_bd}
\end{align}
We rewrite the condition $f_1f_2=\square$ as $f_1=DA^2$ and $f_2=DB^2$ with $(A,B)=1$. Since we are looking for an upper bound for \eqref{to_bd}, we remove the coprimality condition, and we use the bound $\nu(DA^2) \nu(DB^2) \leq \nu(D) \nu(A) \nu(B)$. Then
\begin{align*}
\eqref{to_bd} & \leq  (q^{2g}+q^{\frac{N_{0}s_{0}}{2}})q^{A g/(\log g)^{1/4}} \sum_{\substack{ P|D \Rightarrow \deg(P) \in I_{0} \\ \Omega(D) \leq s_{0}/2\\ \Omega(D) \equiv s_0/2 \pmod 2}} \frac{ |a_{\beta}(D;N_{u})|^2 \nu(D) }{|D|^{1+2\beta}}  \Big( \sum_{\substack{P|A \Rightarrow P \in I_{0} \\ \Omega(A) = (s_{0}/2-\Omega(D))/2}} \frac{ |a_{\beta}(A;N_{u})|^2 \nu(A)}{|A|^{1+2\beta}} \Big)^2\\
& \leq (q^{2g}+q^{\frac{N_{0}s_{0}}{2}})q^{A g/(\log g)^{1/4}} \sum_{\substack{ P|D \Rightarrow \deg(P) \in I_{0} \\ \Omega(D) \leq s_{0}/2 \\ \Omega(D) \equiv s_{0}/2 \pmod 2}} \frac{ |a_{\beta}(D;N_{u})|^2 \nu(D) }{|D|^{1+2\beta}} \frac{1}{ ((s_{0}/2-\Omega(D))/2)!)^2} \\
& \times  \Big( \sum_{P \in I_{0}} \frac{ |a_{\beta}(P;N_{u})|^2}{|P|^{1+2\beta}} \Big)^{\frac{s_{0}}{2}-\Omega(D)}.
\end{align*}
Arranging the polynomials $D$ according to $\Omega(D)$ and using Lemma \ref{power}, we further get that 
\begin{align}
\eqref{to_bd} & \leq  (q^{2g}+q^{\frac{N_{0}s_{0}}{2}}) q^{A g/(\log g)^{1/4}} \Big( \sum_{P \in I_{0}} \frac{ |a_{\beta}(P;N_{u})|^2}{|P|^{1+2\beta}} \Big)^{s_0/2} \sum_{\substack{j=1 \\ j \equiv s_{0}/2 \pmod 2}}^{s_{0}/2} \frac{1}{j!  ((s_{0}/2-j)/2)!)^2} . \label{binom}
\end{align}
Now repeatedly using Stirling's approximation, we have
\begin{align}
\sum_{\substack{j=1 \\ j \equiv s_{0}/2 \pmod 2}}^{s_{0}/2} \frac{1}{j!  ((s_{0}/2-k)/2)!)^2} & \ll \frac{2^{s_{0}/2}}{\sqrt{s_{0}} (s_{0}/2)!} \sum_{j=0}^{\frac{s_0/2-\alpha}{2}} \binom{(s_0/2-\alpha)/2}{j}^2 \ll \frac{2^{s_{0}/2}}{\sqrt{s_{0}} (s_{0}/2)!}  \binom{s_{0}/2-\alpha}{(s_0/2-\alpha)/2} \nonumber  \\
& \ll  \frac{2^{s_{0}}}{ (s_{0}/2)!},\label{bd_binom}
\end{align} 
where in the equation above $\alpha = s_0/2 \pmod 2 \in \{0,1\}$, and where the implied constant is absolute and does not depend on $k$.

Now in equation \eqref{binom}, we use the Prime Polynomial Theorem \eqref{ppt_ineq} for the sum over $P \in  I_0$ and the bound \eqref{single}, and we have
\begin{align*}
\Big( \sum_{P \in I_{0}} &\frac{|a_{\beta}(P;N_u)|^2}{|P|^{1+2\beta}} \Big)^{s_0/2} \leq B(N_u)^{s_0} \Big( \log \frac{1}{N_u \beta} \Big)^{\gamma(N_u) s_0} \Big(\sum_{n=1}^{N_0} \frac{1}{n} \Big)^{s_0/2} \\
& \leq B(N_u)^{s_0} \Big( \log \frac{1}{N_u \beta} \Big)^{\gamma(N_u) s_0} (\log N_0 + 2\gamma)^{s_0/2} \leq  B(N_u)^{s_0} \Big( \log \frac{1}{N_u \beta} \Big)^{\gamma(N_u) s_0}( \log N_0)^{s_0/2} e^{s_0}.
\end{align*}
We further use the fact that
$$B(N_u) \leq B(N_0),$$ and that
$$ \Big( \log \frac{1}{N_u \beta} \Big)^{\gamma(N_u) s_0} \leq \Big( \log \frac{1}{N_0 \beta} \Big)^{\gamma(N_0) s_0}.$$
Combining the two equations above, \eqref{first_bound}, \eqref{binom}, we get (after a possible relabeling of the absolute constant $A$):
\begin{align*}
 \sum_{D \in \mathcal{H}_{2g+1}} & P_{I_{0}}(D;N_{u}) ^{ s_{0}}  \leq (s_{0}/2)! 2^{s_{0}} (q^{2g}+q^{\frac{N_{0}s_{0}}{2}})q^{A g/(\log g)^{1/4}}  (eB(N_0))^{s_{0}}  \Big( \log \frac{1}{N_{0} \beta} \Big)^{\gamma(N_{0})s_{0}} (\log N_{0})^{s_{0}/2}.
\end{align*}
Since $N_{0} s_{0} \leq 4g$, the conclusion follows.
\end{proof}
We also have the following variant of the lemma above, which removes the $q^{Ag/(\log g)^{1/4}}$ term introduced by the use of the large sieve inequality, at the expense of having to choose a shorter Dirichlet polynomial. 

\begin{lem}
\label{variant1}
For $N_0s_0 \leq 2g$ and $s_0$ even, we have
$$\sum_{D \in \mathcal{H}_{2g+1}} P_{I_0}(D;N_u)^{s_0} \leq q^{2g+1} \frac{(s_0)!}{(s_0/2)! 2^{s_0/2}} (eB(N_0))^{s_0} \Big( \log \frac{1}{N_0 \beta} \Big)^{\gamma(N_0)s_0} (\log N_0)^{s_0/2}.$$
\end{lem}
\begin{proof}
The proof is a simplification of the previous proof. Since $P_{I_0}(D;N_u)^{s_0}>0$, we have
\begin{align*}
\sum_{D \in \mathcal{H}_{2g+1}} P_{I_0}(D;N_u)^{s_0} \leq \sum_{D \in \mathcal{M}_{2g+1}} P_{I_0}(D;N_u)^{s_0} = (s_0)! \sum_{D \in \mathcal{M}_{2g+1}} \sum_{\substack{P | f \Rightarrow \deg(P) \in I_{0} \\ \Omega(f) =  s_{0}}} \frac{a_{\beta}(f;N_{u}) \nu(f)\chi_D(f)  }{|f|^{1/2+\beta}}.
\end{align*}
We interchange the sums over $D$ and $f$, and note that since $\Omega(f) = s_0$, we have $\deg(f) \leq N_0 s_0\leq 2g$. Hence, if $f \neq \square$, we have $\sum_{D \in \mathcal{M}_{2g+1}} \chi_D(f) =0$. Using the fact that $\nu(f^2) \leq \nu(f)/2^{\Omega(f)}$, it follows that
 \begin{align*}
\sum_{D \in \mathcal{H}_{2g+1}} P_{I_0}(D;N_u)^{s_0} & \leq \sum_{D \in \mathcal{M}_{2g+1}} P_{I_0}(D;N_u)^{s_0} \leq q^{2g+1} (s_0)!  \sum_{\substack{P | f \Rightarrow \deg(P) \in I_{0} \\ \Omega(f) =  s_{0}/2}} \frac{|a_{\beta}(f;N_{u})|^2 \nu(f) }{2^{\Omega(f)}|f|^{1+2\beta}} \\
&= q^{2g+1} \frac{(s_0)!}{(s_0/2)! 2^{s_0/2}} \Big( \sum_{P \in I_0} \frac{|a_{\beta}(P;N_u)|^2}{|P|^{1+2\beta}} \Big)^{s_0/2}.
\end{align*}
Now we treat the sum over $P \in I_0$ similarly as in the proof of Lemma \ref{sum_primes}, and the conclusion follows.
\end{proof}

We also need the following two lemmas, the second of which is a variant of the first. The first lemma uses the large sieve inequality; it has the advantage that it allows one to choose a longer Dirichlet polynomial, but it introduces an extra term of the form $q^{Ag/(\log g)^{1/4}}$ in the upper bound. The second lemma uses a simpler orthogonality of characters argument and removes the $q^{Ag/(\log g)^{1/4}}$ term, but only allows for shorter polynomials. 

\begin{lem} \label{j+1}
For $0 \leq j<K$, let $\ell_j$ be even parameters, and let $s_{j+1}$ be even such that $2 \sum_{h \leq j} \ell_h N_h+s_{j+1}N_{j+1} \leq 4g$. Let $j<u \leq K$. 
Then we have
\begin{align*}
\sum_{D \in \mathcal{H}_{2g+1}} & \prod_{h =0}^j  E_{\ell_{h}} \Big(  k P_{I_{h}}(D; N_{j})\Big) \Big(P_{I_{j+1}}(D;N_{u}) \Big)^{s_{j+1}}  \leq q^{2g+A g/(\log g)^{1/4}} \exp \Big( O\Big(k^2 \Big( \log \frac{1}{N_j \beta} \Big)^{2 \gamma(N_j)} \Big) \\
& \times (s_{j+1}/2)! \Big(2\sqrt{\log \frac{N_{j+1}}{N_j}}\Big)^{s_{j+1}}B(N_{j+1})^{s_{j+1}} \Big( \log \frac{1}{N_{j+1} \beta} \Big)^{\gamma(N_{j+1})s_{j+1}}  N_{j}^{3k^2 B(N_{j})^2 \Big( \log \frac{1}{N_{j} \beta} \Big)^{2 \gamma(N_{j})}},
\end{align*}
for some absolute constant $A>0$.
\end{lem}
\begin{proof}
Let $J \subset \{0,\ldots,j\}$ be the subset of indices $h$ such that $E_{\ell_{h}}(k P_{I_{h}}(D;N_{j})) >1$. Then
\begin{align}
 \sum_{D \in \mathcal{H}_{2g+1}} &  \prod_{h =0}^j  E_{\ell_{h}} \Big(  k P_{I_{h}}(D; N_{j})\Big) \Big(P_{I_{j+1}}(D;N_{u}) \Big)^{s_{j+1}} \label{gd} \\
 & \leq \sum_{D \in \mathcal{H}_{2g+1}} \prod_{h \in J} E_{\ell_{h}}^2 \Big(  k P_{I_{h}}(D; N_{j})\Big) \Big(P_{I_{j+1}}(D;N_{u}) \Big)^{s_{j+1}} \nonumber \\
 &=(s_{j+1}/2)!^2 \sum_{D \in \mathcal{H}_{2g+1}} \Bigg(\prod_{h \in J}  \Big( \sum_{\substack{P|f_h \Rightarrow P \in I_{h} \\ \Omega(f_h) \leq \ell_{h}}} \frac{k^{\Omega(f_h)} a(f_h;N_{j}) \nu(f_h) \chi_D(f_h)}{|f_h|^{1/2+\beta}} \Big) \nonumber \\
 & \times \sum_{\substack{P|f_{j+1} \Rightarrow P \in I_{j+1} \\ \Omega(f_{j+1}) = s_{j+1}/2}} \frac{a(f_{j+1};N_{u})\nu(f_{j+1}) \chi_D(f_{j+1})}{|f_{j+1}|^{1/2+\beta}} \Bigg)^2 \nonumber  \\
 &= (s_{j+1}/2)!^2 \sum_{D \in \mathcal{H}_{2g+1}} \Bigg(\sum_{\substack{P | B \Rightarrow \deg(P) \leq N_{j+1} \\ \deg(B) \leq \sum_{h \in J} \ell_{h} N_{h} + s_{j+1} N_{j+1}/2}}  \frac{\nu(B)c(B)  \chi_D(B)}{|B|^{1/2+\beta}}  \Bigg)^2, \nonumber
 \end{align}
 where
 $$c(B) = \sum_{\substack{ B = (\prod_{h \in J} f_h) f_{j+1} \\ \Omega(f_h) \leq \ell_{h} \\ \Omega(f_{j+1})=s_{j+1}/2}} k^{ \sum_{h \in J} \Omega(f_h)} ( \prod_{h \in J} a(f_h;N_{j})) a(f_{j+1};N_{u}).$$
 Using the large sieve inequality in Lemma \ref{large_sieve}, we get that
 \begin{align*}
 \eqref{gd} \leq (s_{j+1}/2)!^2 (q^{2g} + q^{\sum_{h \leq j} \ell_{h} N_{h}+s_{j+1} N_{j+1}/2})q^{A g/(\log g)^{1/4}} \sum_{B_1B_2=\square}  \frac{\nu(B_1) \nu(B_2) |c(B_1)c(B_2)|}{|B_1B_2|^{1/2+\beta}}.
 \end{align*}
We write $B_1=( \prod_{h \in J} A_h) A_{j+1}$ and $B_2=( \prod_{h \in J} C_h) C_{j+1}$. Note that the condition $B_1 B_2=\square$ is equivalent to $A_h C_h = \square$ for $h \leq j$ and $A_{j+1} C_{j+1} = \square$, since the $A_h$ and $A_{j+1}$ are pairwise coprime (and the same holds for $C_h$ and $C_{j+1}$). Using the condition that $2 \sum_{h \leq j} \ell_h N_h+s_{j+1} N_{j+1} \leq 4g$, we get that
 \begin{align}
& \eqref{gd}  \leq (s_{j+1}/2)!^2  q^{2g+Ag/(\log g)^{1/4}} \sum_{\substack{ P|A_{j+1} C_{j+1} \Rightarrow P \in I_{j+1} \\ A_{j+1} C_{j+1} = \square  \\ \Omega(A_{j+1})=\Omega(C_{j+1} )=s_{j+1}/2}} \frac{ | a(A_{j+1} C_{j+1};N_{u})|  \nu(A_{j+1}) \nu(C_{j+1})}{  |A_{j+1} C_{j+1}|^{1/2+\beta}}\label{af}\\
 & \times  \prod_{h \in J} \sum_{\substack{P|A_h C_h \Rightarrow P \in I_{h} \\  A_h C_h  = \square \\ \Omega(A_h), \Omega(C_h) \leq \ell_h}} \frac{  k^{\Omega(A_h)+\Omega(C_h) }  |a(A_h C_h ;N_{j}) | \nu(A_h) \nu(C_h) }{  |A_h C_h|^{1/2+\beta} }. \nonumber 
 \end{align}
 The conditions $A_h C_h=\square$ and $A_{j+1} C_{j+1} = \square$ can be rewritten as $A_h \mapsto D_h A_h^2, C_h \mapsto D_h C_h^2$, with $(A_h,C_h)=1$ and $A_{j+1} \mapsto D_{j+1} A_{j+1}^2, C_{j+1} \mapsto D_{j+1} C_{j+1}^2$ with $(A_{j+1},C_{j+1})=1$. Removing the coprimality conditions and using the bound $\nu(fh^2) \leq \nu(f) \nu(h)$ for any polynomials $f,h$ and $\nu(f) \leq 1$ for any $f$, we get that
 \begin{align*}
 & \eqref{gd}  \ll (s_{j+1}/2)!^2  q^{2g+A g/(\log g)^{1/4}}\sum_{\substack{ P|A_{j+1} C_{j+1} D_{j+1} \Rightarrow P \in I_{j+1} \\  \Omega(A_{j+1}^2D_{j+1})=\Omega(C_{j+1}^2D_{j+1} )=s_{j+1}/2}}\\
&  \frac{| a(A_{j+1} C_{j+1}D_{j+1};N_{u})|^2  \nu(A_{j+1}) \nu(C_{j+1})\nu(D_{j+1})}{  |A_{j+1} C_{j+1}D_{j+1}|^{1+2\beta}} \\
 & \times  \prod_{h \in J} \Bigg( \sum_{\substack{P|A_h C_h D_h\Rightarrow P \in I_{h} \\  \Omega(D_hA_h^2),\Omega(D_h C_h^2) \leq\ell_{h} }}\frac{  k^{2\Omega(A_h)+2\Omega(C_h)+2 \Omega(D_h) }  |a(A_h C_h D_h ;N_j)|^2  \nu(A_h) \nu(C_h) \nu(D_h) }{ |A_h C_h D_h|^{1+2\beta} }\Bigg)  .
 \end{align*}
 For the product over $h \in J$, we trivially bound the sums over $A_h, C_h, D_h$ as follows:
 \begin{align}
  \sum_{\substack{P|A_h C_h D_h \Rightarrow P \in I_{h} \\  \Omega(D_hA_h^2),\Omega(D_h C_h^2) \leq \ell_{h} }} & \frac{  k^{2\Omega(A_h)+2\Omega(C_h)+2 \Omega(D_h) }  |a(A_h C_h D_h ;N_{j})|^2  \nu(A_h) \nu(C_h) \nu(D_h) }{  |A_h C_h D_h|^{1+2\beta} } \label{ae} \\
 & \leq  \Big( \sum_{P|A_h  \Rightarrow P \in I_{h} }\frac{  k^{2\Omega(A_h) }  |a(A_h  ;N_{j})|^2  \nu(A_h)  }{  |A_h |^{1+2\beta} } \Big)^3 \nonumber \\
 & \leq \Big( \sum_{P|A_h  \Rightarrow P \in I_{h} }\frac{  k^{2\Omega(A_h) }  B(N_j)^{2\Omega(f)} \Big( \log \frac{1}{N_j \beta}\Big)^{2\gamma(N_j)\Omega(f)}   }{  |A_h |^{1+2\beta} } \Big)^3\nonumber  \\
 &= \prod_{P \in I_{h}} \Big(1- \frac{k^2B(N_j)^2 \Big( \log \frac{1}{N_{j} \beta} \Big)^{2 \gamma(N_{j})}}{|P|^{1+2\beta}} \Big)^{-3},\nonumber 
 \end{align}
 where we have used inequality \eqref{combined} in the third line. 
 To deal with the sum over $A_{j+1}, C_{j+1}$ we proceed as in the proof of Lemma \ref{sum_primes} (see equations \eqref{binom} and \eqref{bd_binom}), and it follows that
 \begin{align}
  \sum_{\substack{ P|A_{j+1} C_{j+1} D_{j+1} \Rightarrow P \in I_{j+1} \\  \Omega(A_{j+1}^2D_{j+1})=\Omega(C_{j+1}^2D_{j+1} )=s_{j+1}/2}} & \frac{ a(A_{j+1} C_{j+1} D_{j+1};N_{u})^2  \nu(A_{j+1}) \nu(C_{j+1})\nu(D_{j+1})}{  |A_{j+1} C_{j+1}D_{j+1}|^{1+2\beta}} \ll \frac{2^{s_{j+1}}}{(s_{j+1}/2)!} \nonumber  \\
 & \times  \Big( \sum_{P \in I_{j+1}} \frac{|a(P;N_{u})|^2}{|P|^{1+2\beta}}\Big)^{s_{j+1}/2},\label{ad}
 \end{align}
 where the implied constant above is absolute and does not depend on $k$.
 For the sum over $P \in I_{j+1}$, we use the Prime Polynomial Theorem, and \eqref{single}. 
  Combining \eqref{af}, \eqref{ae} and \eqref{ad} and using the facts that $B(N_u) \leq B(N_{j+1})$ and that 
 $$ \Big( \log \frac{1}{N_u \beta} \Big)^{\gamma(N_u)} \leq  \Big( \log \frac{1}{N_{j+1} \beta} \Big)^{\gamma(N_{j+1})},$$
  it follows that
 \begin{align*}
 \eqref{gd} & \leq q^{2g+Ag/(\log g)^{1/4}} (s_{j+1}/2)! \Big(2\sqrt{ \frac{N_{j+1}}{N_j}} B(N_{j+1})\Big)^{s_{j+1}} \prod_{h \in J} \prod_{P \in I_h} \Big(1- \frac{k^2B(N_{j})^2 \Big( \log \frac{1}{N_{j} \beta} \Big)^{2 \gamma(N_{j})}}{|P|^{1+2\beta}} \Big)^{-3} \\
 & \times  \Big( \log \frac{1}{N_{j+1} \beta} \Big)^{\gamma(N_{j+1})s_{j+1}} . 
 \end{align*}
  Now we have that
 \begin{align*}
 \prod_{h \in J} \prod_{P \in I_h}& \Big(1- \frac{k^2B(N_{j})^2 \Big( \log \frac{1}{N_{j} \beta} \Big)^{2 \gamma(N_{j})}}{|P|^{1+2\beta}} \Big)^{-3} \leq 
\prod_{\deg(P) \leq N_{j}} \Big(1- \frac{k^2B(N_{j})^2 \Big( \log \frac{1}{N_{j} \beta} \Big)^{2 \gamma(N_{j})}}{|P|^{1+2\beta}} \Big)^{-3} \\
&= N_{j}^{3k^2 B(N_{j})^2 \Big( \log \frac{1}{N_{j} \beta} \Big)^{2 \gamma(N_{j})}}  \exp \Big( O\Big(k^2 \Big( \log \frac{1}{N_j \beta} \Big)^{2 \gamma(N_j)} \Big) \Big),
\end{align*}
where the implied constant in the exponential term does not depend on $k$, and where we have used Mertens' theorem over function fields (see Lemma $3.6$ in \cite{hybrid}). Then we get that
\begin{align*}
\eqref{gd} & \leq \exp \Big( O\Big(k^2 \Big( \log \frac{1}{N_j \beta} \Big)^2 \Big)\Big)  q^{2g+A g/(\log g)^{1/4}} (s_{j+1}/2)! 
\Big(2\sqrt{ \frac{N_{j+1}}{N_j}} B(N_{j+1})\Big)^{s_{j+1}} \\
& \times \Big( \log \frac{1}{N_{j+1} \beta} \Big)^{\gamma(N_{j+1})s_{j+1}}  N_{j}^{3k^2 B(N_{j})^2 \Big( \log \frac{1}{N_{j} \beta} \Big)^{2 \gamma(N_{j})}} .
 \end{align*}
 \end{proof}
 
 A simplification of the argument above yields the following lemma.
 \begin{lem}
 \label{variant2}
For $0 \leq j<K$, let $\ell_j$ be even parameters, and let $s_{j+1}$ be even such that $ \sum_{h \leq j} \ell_h N_h+s_{j+1}N_{j+1} \leq 2g$. Let $j<u \leq K$. Then we have
\begin{align*}
\sum_{D \in \mathcal{H}_{2g+1}}  \prod_{h =0}^j  E_{\ell_{h}} \Big(  k P_{I_{h}}(D; N_{j})\Big) \Big(P_{I_{j+1}}(D;N_{u}) \Big)^{s_{j+1}}  \leq q^{2g+1} \exp \Big( O\Big(k^2 \Big( \log \frac{1}{N_j \beta} \Big)^{2 \gamma(N_j)} \Big)  \\
 \times  \frac{(s_{j+1})!}{(s_{j+1}/2)!2^{s_{j+1}/2}} \Big(\sqrt{ \log \frac{N_{j+1}}{N_j}} B(N_{j+1})\Big)^{s_{j+1}}\Big( \log \frac{1}{N_{j+1} \beta} \Big)^{\gamma(N_{j+1}) s_{j+1}} N_{j}^{\frac{k^2B(N_{j})^2}{2} \Big( \log \frac{1}{N_{j} \beta} \Big)^{2\gamma(N_{j})} }.
\end{align*}
%where the implied constant does not depend on $k$.
 \end{lem}
 \begin{proof}
 Since the summands are positive (because $\ell_h$ and $s_{j+1}$ are even, see the explanation right before equation \eqref{taylor}), we have that
 $$\sum_{D \in \mathcal{H}_{2g+1}}  \prod_{h =0}^j  E_{\ell_{h}} \Big(  k P_{I_{h}}(D; N_{j})\Big) \Big(P_{I_{j+1}}(D;N_{u}) \Big)^{s_{j+1}} \leq \sum_{D \in \mathcal{M}_{2g+1}} \prod_{h =0}^j  E_{\ell_{h}} \Big(  k P_{I_{h}}(D; N_{j})\Big) \Big(P_{I_{j+1}}(D;N_{u}) \Big)^{s_{j+1}} ,$$
 and
 \begin{align}
 \sum_{D \in \mathcal{M}_{2g+1}} & \prod_{h =0}^j  E_{\ell_{h}} \Big(  k P_{I_{h}}(D; N_{j})\Big) \Big(P_{I_{j+1}}(D;N_{u}) \Big)^{s_{j+1}} = (s_{j+1})! \nonumber   \\
 & \times  \sum_{D \in \mathcal{M}_{2g+1}} \prod_{h=0}^{j} \bigg( \sum_{\substack{P | f_h \Rightarrow \deg(P) \in I_{h} \\ \Omega(f_h) \leq \ell_{h}}} \frac{k^{\Omega(f_h)}  a_{\beta}(f_h;N_{j}) \nu(f_h)\chi_D(f_h)}{|f_h|^{1/2+\beta}} \bigg) \nonumber  \\
 & \times \Big(\sum_{\substack{P|f_{j+1} \Rightarrow \deg(P) \in I_{j+1} \\ \Omega(f_{j+1})=s_{j+1}}} \frac{a_{\beta}(f_{j+1};N_u) \nu(f_{j+1}) \chi_D(f_{j+1})}{|f_{j+1}|^{1/2+\beta}} \Big) \label{to_bd4}
 \end{align}
 Interchanging the sums over $D$ and $f_h$ and $f_{j+1}$, note that if $f_0 \ldots f_{j+1} \neq \square$, then 
 $$\sum_{D \in \mathcal{M}_{2g+1}} \chi_D( f_0 \ldots f_{j+1}) =0,$$ since $\deg(f_0 \ldots f_{j+1}) \leq \sum_{h=0}^j \ell_h N_h+s_{j+1} N_{j+1} \leq 2g+1$. It follows that in \eqref{to_bd4}, we need $f_0 \ldots f_{j+1} = \square$, and since the $f_i$ are pairwise coprime, this happens if and only if each $f_h =\square$, for $h \leq j+1$. Bounding $\nu(f_h)^2 \leq \nu(f_h)/2^{\Omega(f_j)}$ and using the bound \eqref{combined}, we get that
 \begin{align*}
&  \sum_{D \in \mathcal{M}_{2g+1}}  \prod_{h =0}^j  E_{\ell_{h}} \Big(  k P_{I_{h}}(D; N_{j})\Big) \Big(P_{I_{j+1}}(D;N_{u}) \Big)^{s_{j+1}}  \leq q^{2g+1} (s_{j+1})! \\
 & \times  \prod_{h=0}^j \bigg( \sum_{\substack{P | f_h \Rightarrow \deg(P) \in I_{h} \\ \Omega(f_h) \leq \ell_{h}/2}} \frac{k^{2\Omega(f_h)} B(N_{j})^{2 \Omega(f_h)} \Big( \log \frac{1}{N_{j} \beta} \Big)^{2\gamma(N_{j}) \Omega(f_h)} }{2^{\Omega(f_h)}|f_h|^{1+2\beta}} \bigg)\\
 & \times  \Big( \sum_{\substack{P|f_{j+1} \Rightarrow P \in I_{j+1} \\ \Omega(f_{j+1})=s_{j+1}/2}} \frac{ B(N_u)^{2 \Omega(f_{j+1})} \Big( \log \frac{1}{N_u \beta} \Big)^{2 \gamma(N_u) \Omega(f_{j+1})} \nu(f_{j+1})}{2^{\Omega(f_{j+1})}|f_{j+1}|^{1+2\beta}} \Big)\nonumber  \\
& \leq  q^{2g+1} \frac{(s_{j+1})!}{(s_{j+1}/2)! 2^{s_{j+1}/2}} B(N_{j+1})^{s_{j+1}} \Big( \log \frac{1}{N_{j+1} \beta} \Big)^{\gamma(N_{j+1}) s_{j+1}}\\
& \times \prod_{h=0}^{j} \Bigg( \prod_{\deg(P) \in I_{h}} \Big(1- \frac{k^2B(N_{j+1})^2\Big( \log \frac{1}{N_{j+1} \beta} \Big)^{2\gamma(N_{j+1})}  }{2|P|^{1+2\beta}} \Big)^{-1} \Bigg) \Bigg( \sum_{P \in I_{j+1}} \frac{1}{|P|^{1+2\beta}} \Bigg)^{s_{j+1}/2} \nonumber .
 \end{align*}
 For the sum over $P \in I_{j+1}$ we use the Prime Polynomial Theorem \eqref{ppt} and we have
 
 \begin{align*}
\sum_{D \in \mathcal{M}_{2g+1}}&  \prod_{h =0}^j  E_{\ell_{h}} \Big(  k P_{I_{h}}(D; N_{j})\Big) \Big(P_{I_{j+1}}(D;N_{u}) \Big)^{s_{j+1}}  \leq q^{2g+1} \exp \Big( O\Big(k^2 \Big( \log \frac{1}{N_j \beta} \Big)^{2 \gamma(N_j)} \Big)\\
& \times   \frac{(s_{j+1})!\Big(\sqrt{\log \frac{N_{j+1}}{N_j}} B(N_{j+1})\Big)^{s_{j+1}}\Big( \log \frac{1}{N_{j+1} \beta} \Big)^{\gamma(N_{j+1}) s_{j+1}}}{(s_{j+1}/2)!2^{s_{j+1}/2}}  N_{j}^{\frac{k^2B(N_{j})^2}{2} \Big( \log \frac{1}{N_{j} \beta} \Big)^{2\gamma(N_{j})} }.
 \end{align*}
% where the implied constant does not depend on $k$.
 \end{proof}
 A simplification of the argument above (when there is no contribution from the $j+1$ interval) yields the following lemma whose proof we omit because it follows from the previous lemma.
 \begin{lem} \label{sm1}
For $0 \leq j \leq K$, let $\ell_j$ be even parameters such that $\sum_{h \leq K} \ell_h N_h \leq 2g$. Then we have
\begin{align*}
\sum_{D \in \mathcal{H}_{2g+1}}  \prod_{h =0}^{K}  E_{\ell_{h}} \Big(  k P_{I_{h}}(D; N_{K})\Big)  \leq   q^{2g+1} \exp \Big( O\Big(k^2 \Big( \log \frac{1}{N_K \beta} \Big)^{2 \gamma(N_K)} \Big)  N_{K}^{\frac{k^2B(N_{K})^2}{2} \Big( \log \frac{1}{N_{K} \beta} \Big)^{2\gamma(N_{K})} }.
\end{align*}
 \end{lem}

\section{Proof of Theorem \ref{theorem_ub}}
\label{proof_mainthm}
Here, we begin the proof of Theorem \ref{theorem_ub} and consider different ranges for $\beta$. 
\subsection{The range $\beta \gg g^{-\frac{1}{k}+\epsilon}$, first step}
Note that a sharp upper bound was obtained in the range $\beta \gg g^{-\frac{1}{2k}+\epsilon}$ in \cite{ratios}. We then assume that $g^{-\frac{1}{k}+\epsilon} \ll \beta \ll g^{-\frac{1}{2k}+\epsilon}$. 

In what follows, the absolute constant $A$ might change from line to line.

We first assume that $k \geq 1$. 
In this case, we choose the parameters as follows: 
\begin{equation}
N_{0} = \Big[ \frac{ 4\log_q g (d-1/2)}{k\alpha(1+\epsilon)+2 \epsilon} \Big], \, s_{0} = 2 \Big[ \frac{2g}{N_{0}} \Big], \, \ell_{0} = 2\lceil s_{0}^d/2 \rceil,\label{0choice}
\end{equation} 
for some $1/2<d<1$.
For $1 \leq j \leq K$, we also choose 
\begin{equation}
N_{j} =[ r (N_{j-1}+1) ], \, s_{j} = 2 \Big[  \frac{ag}{2N_{j}}\Big], \, \ell_{j} = 2 \lceil s_{j}^d/2 \rceil, \label{n1j}
\end{equation}
for some constants $a< 4, r>1$. We choose $a, d,r$ such that
\begin{equation}
\label{t_choice}
\frac{a(d-1/2)}{r} = 2-\frac{k\epsilon}{2}.
\end{equation}
For example, we can choose
\begin{equation}
 a = 4-\Big(\frac{k\epsilon}{4}\Big)^2, \, d = \frac{16+k\epsilon}{16+2k\epsilon}, \,  r= \frac{8-k\epsilon}{8-2k\epsilon}.
 \label{adr}
 \end{equation}
 %\acom{$a=4, d= \frac{4-3k\epsilon/2}{4-k\epsilon}, r=\frac{2}{2-k\epsilon/2}$.}
We choose $K$ such that $N_K$ is the largest integer of the form given in \eqref{n1j} for which
\begin{equation}
N_K  \leq [k (\log g)^{\frac{5}{4}}]-1,
\label{nk3}
\end{equation}
%\acom{$$N_K = [k (\log g)^{\frac{5}{4}}]-1.$$}
%where $A$ is the parameter in the Lemmas \ref{sum_primes} and \ref{j+1}. 
Note that the conditions in Lemmas \ref{sum_primes}, \ref{j+1} and \ref{sm1} are satisfied with the above choices of parameters.
%\begin{equation}
%N_{K} = [cg], \label{n1k1}
%\end{equation}
%\acom{$$N_K = \Big[ \frac{(\log g)^{5/4} (\rat-2k\alpha)}{2A} \Big].$$}
%for $c$ a small constant such that
%$$ 4^d g^{d-1} N_0^{1-d} + \frac{c^{1-d} a^d}{r^{1-d}-1} \leq 4-a.$$
%\acom{$$a^dc^{1-d}r^{1-d} \leq \epsilon (r^{1-d}-1).$$}
%The condition above ensures that \eqref{second}, \eqref{third} are satisfied.

We now use Lemma \ref{lem_initial}. If $D \notin \mathcal{T}_{0}$, then there exists some $0 \leq u \leq K$ such that 
$$ 1< \Big(  \frac{ke^2}{\ell_{0}} P_{I_{0}} (D;N_{u}) \Big)^{s_{0}},$$ and we then get that 
\begin{align*} 
\sum_{D \notin \mathcal{T}_{0}} &  \frac{1}{|L(1/2+\beta+it,\chi_D)|^k} \leq \sum_{D \in \mathcal{H}_{2g+1}} \frac{1}{|L(1/2+\beta+it,\chi_D)|^k}  \Big( \frac{ke^2}{\ell_{0}} P_{I_{0}}(D;N_{u}) \Big)^{s_{0}}.  
%& \leq \Big( \frac{ke^2}{\ell_{m,0}} \Big)^{s_{m,0}} \bigg( \sum_{D \in \mathcal{H}_{2g+1}}  \frac{1}{|L(1/2+\beta+it,\chi_D)|^{2k}} \bigg)^{1/2}  \bigg( \sum_{D \in \mathcal{H}_{2g+1}}  P_{I_{1,0}} (D; N_{1,u}) ^{2 s_{1,0}} \bigg)^{1/2}.\nonumber 
\end{align*}
We use the pointwise bound in Lemma \ref{pointwise} for the $L$--function, and then
\begin{align*}
\sum_{D \notin \mathcal{T}_{0}} &  \frac{1}{|L(1/2+\beta+it,\chi_D)|^k} \leq q^{gk \alpha(1+\epsilon)}  \Big( \frac{ke^2}{\ell_{0}} \Big)^{s_{0}} \sum_{D \in \mathcal{H}_{2g+1}} \Big( P_{I_{0}}(D;N_{u}) \Big)^{s_{0}}.
\end{align*}
Now using Lemma \ref{sum_primes}, we get that 
\begin{align*}
\sum_{D \notin \mathcal{T}_{0}}  \frac{1}{|L(1/2+\beta+it,\chi_D)|^k} & \leq q^{2g+gk \alpha(1+\epsilon)+A g/(\log g)^{1/4}}  \Big( \frac{ke^2}{\ell_{0}} \Big)^{s_{0}} (s_{0}/2)! (2e)^{s_{0}}  B(N_{0})^{s_{0}}  \\
& \times  \Big( \log \frac{1}{N_{0} \beta} \Big)^{\gamma(N_{0})s_{0}} (\log N_{0})^{s_{0}/2}.
\end{align*}
Using Stirling's formula and the expression for $\ell_{0}$ (equation \eqref{0choice}), we get that
\begin{align}
& \sum_{D \notin \mathcal{T}_{0}}  \frac{1}{|L(1/2+\beta+it,\chi_D)|^k}  \leq q^{2g+gk \alpha(1+\epsilon)+Ag/(\log g)^{1/4}} \exp \Big(-s_{0} (d-1/2) \log s_{0}  \Big) \nonumber \\
& \times \exp \Big(s_{0} \log \Big( 2^{1/2} e^{5/2} k B(N_{0})  \Big( \log  \frac{1}{N_{0} \beta} \Big)^{\gamma(N_{0})} \sqrt{ \log N_{0}}  \Big) \Big).
\label{first_contrib}
\end{align}
With the choice of parameters \eqref{0choice} it follows that
\begin{equation}
\sum_{D \notin \mathcal{T}_{0}}  \frac{1}{|L(1/2+\beta+it,\chi_D)|^k} = o(q^{2g}). \label{s10}
\end{equation}
Now we consider the contribution from $D \in \mathcal{T}_{0}$. 
Using Lemma \ref{lem_initial}, it follows that
\begin{align*}
\sum_{D \in \mathcal{T}_{0}}  \frac{1}{|L(1/2+\beta+it,\chi_D)|^k}  \leq \exp (O(k)) \Big(\sum_{D \in \mathcal{T}_{0}} S_{1}(D) + \sum_{D \in \mathcal{T}_{0}} S_{2}(D)\Big) .
\end{align*}
Now using Lemmas \ref{lem_initial} and \ref{sm1}, we have that
\begin{align*}
\sum_{D \in \mathcal{T}_{0}}& S_{1}(D) \leq q^{2g}  \Big(1-q^{-(N_{K}+1)\beta} \Big)^{-\frac{2gk}{N_{K}+1}}  (\log g)^{k/2} N_{K}^{\frac{k}{2}+\frac{k^2B(N_{K})^2}{2} \Big( \log \frac{1}{N_{K} \beta} \Big)^{2\gamma(N_{K})} } \\
& \times  \exp \Big( O \Big( k^2 \Big( \log \frac{1}{N_K \beta} \Big)^{2 \gamma(N_K)} \Big) \Big).
\end{align*}
Since $\beta \ll g^{-\frac{1}{2k}+\epsilon}$, it follows that $N_K \beta \to 0$, so $\gamma(N_K)=1$. Now we use the expression \eqref{b*} for $B(N_K)$ and the expression \eqref{nk3}. Recall that  $\beta \gg g^{-\frac{1}{k}+\epsilon}$ and $k \geq 1$. Then $\log(1/\beta) \leq \log g$, and 
$$ \Big(1-q^{-(N_{K}+1)\beta} \Big)^{-\frac{2gk}{N_{K}+1}}  \leq \exp \Big( \frac{2g}{(\log g)^{5/4}} \log \frac{1}{\beta (\log g)^{5/4}} \Big)  \leq q^{2g/(\log g)^{1/4}}.  $$ It follows that
\begin{equation}
\sum_{D \in \mathcal{T}_{0}} S_{1}(D) \leq q^{2g+2g/(\log g)^{1/4}} \exp (k^2  (\log g)^2 \log \log g). \label{s11d}
\end{equation} 
Now we consider the contribution from $S_{2}(D)$. Using Lemmas \ref{lem_initial} and Lemma \ref{j+1} and since $r<2$, we have that 
\begin{align*}
& \sum_{D \in \mathcal{T}_{0}} S_{2}(D) \leq q^{2g+A g/(\log g)^{1/4}}  (\log g)^{k/2} \sum_{j=0}^{K-1} (K-j) \exp \Big( O \Big( k^2 \Big( \log \frac{1}{N_j \beta} \Big)^{2 \gamma(N_j)} \Big) \Big)  \Big(1-q^{-(N_{j}+1)\beta} \Big)^{-\frac{2gk}{N_{j}+1}} \\
& \times  N_{j}^{k/2+3k^2 B(N_{j})^2 \Big( \log \frac{1}{N_{j} \beta} \Big)^{2 \gamma(N_{j})}} \Big( \frac{ke^2}{\ell_{j+1}} \Big)^{s_{j+1}}(s_{j+1}/2)! 2^{s_{j+1}}B(N_{j+1})^{s_{j+1}} \Big( \log \frac{1}{N_{j+1} \beta} \Big)^{\gamma(N_{j+1})s_{j+1}}  .
\end{align*}
Using Stirling's formula, we get (similarly as in \cite{ratios}):
\begin{align*}
\sum_{D \in \mathcal{T}_{0}} &S_{2}(D) \leq q^{2g+ A g/(\log g)^{1/4}}  (\log g)^{k/2}   \sum_{j=0}^{K-1} (K-j)\sqrt{s_{j+1}}  \exp \Big( \frac{2gk}{N_j} \log(1/\beta) -\frac{2gk}{N_j } \log (N_j) \\
&-(d-1/2) s_{j+1} \log s_{j+1} +s_{j+1} \log \Big(2^{1/2} e^{3/2} k B(N_{j+1}) \Big( \log \frac{1}{ N_{j+1}\beta} \Big)^{\gamma(N_{j+1})}\Big)\nonumber  \\
& + (\log N_{j}) \Big(3B(N_{j})^2k^2(\log \frac{1}{N_{j} \beta} )^{2\gamma(N_{j})} +k/2 \Big) \Big)  \exp \Big( O \Big( k^2 \Big( \log \frac{1}{N_j \beta} \Big)^{2 \gamma(N_j)}\Big).
\end{align*}
Now using formulas \eqref{n1j}, we have
$$s_{j+1} = \frac{ag}{r N_j} + O \Big( \frac{g}{N_j^2}+1 \Big),$$ and
$$\log s_{j+1} = \log \frac{g}{N_j} + O(1).$$
Hence it follows that 
\begin{align}
\sum_{D \in \mathcal{T}_{0}} &S_{2}(D) \leq q^{2g+ A g/(\log g)^{1/4}}  (\log g)^{k/2}   \sum_{j=0}^{K-1} (K-j)\sqrt{s_{j+1}} \exp \Big(  \frac{g \log g}{N_{j}} \Big(2k \alpha- \rat\Big) \label{j+1_bd} \\
&+ \frac{g \log N_{j}}{N_{j}} \Big( \rat-2k \Big) + \frac{ag}{r N_{j}} \log \Big( 2^{1/2} e^{3/2}c k B(N_{j+1}) \Big( \log \frac{1}{ N_{j+1}\beta} \Big)^{\gamma(N_{j+1})}\Big)\nonumber  \\
& + (\log N_{j}) \Big(3B(N_{j})^2k^2(\log \frac{1}{N_{j} \beta} )^{2\gamma(N_{j})} +k/2 \Big) \Big)  \exp \Big( O \Big( k^2 \Big( \log \frac{1}{N_j \beta} \Big)^{2 \gamma(N_j)}\Big), \nonumber 
\end{align}
for some absolute constant $c$ (not depending on $k$).
%Since $2k \alpha \leq 2 -2k\epsilon$ and keeping in mind the choice of parameters, it follows that (after a relabeling of the $\epsilon$ and of the constant $A$):
%\begin{align}
%\sum_{D \in \mathcal{T}_{0}} &S_{2}(D) \ll q^{2g+Ag/(\log g)^{1/4}}.
%\label{s12d}
%\end{align}

Since $\beta \gg g^{-\frac{1}{k}+\epsilon}$, and given \eqref{t_choice}, we have $2k\alpha - \rat \leq -k\epsilon$. Since $k \geq 1$, we also have $\rat-2k \leq  -k\epsilon/2$, and hence the expression above is increasing as a function of $j$. Then we have:
\begin{align*}
\sum_{D \in \mathcal{T}_{0}} &S_{2}(D) \leq q^{2g+ A g/(\log g)^{1/4}} K (\log g)^{k/2}  \exp \Big(  \frac{g \log g}{N_{K-1}} \Big(2k \alpha- \rat\Big) \\
&+ \frac{g \log N_{K-1}}{N_{K-1}} \Big( \rat-2k \Big) + \frac{ag}{r N_{K-1}} \log \Big(2^{1/2} e^{3/2} c k B(N_{K}) \Big( \log \frac{1}{ N_{K}\beta} \Big)^{\gamma(N_{K})}\Big)\nonumber  \\
& + (\log N_{K-1}) \Big(3B(N_{K-1})^2k^2(\log \frac{1}{N_{K-1} \beta} )^{2\gamma(N_{K-1})} +k/2 \Big) \Big)  \exp \Big( O \Big( k^2 \Big( \log \frac{1}{ \beta} \Big)^{2}\Big).
\end{align*}
Given \eqref{0choice} and \eqref{nk3}, we have that $K \asymp \log \log g$. Note that the term involving $N_{K-1}$ is negative, so we can bound 
\begin{equation}
\sum_{D \in \mathcal{T}_0} S_2(D) \leq q^{2g+Ag/(\log g)^{1/4}}.
\label{s12d}
\end{equation}

Combining the bounds \eqref{s10}, \eqref{s11d}, \eqref{s12d} leads to
\begin{equation}
\label{to_improve}
\sum_{D \in \mathcal{H}_{2g+1}}   \frac{1}{|L(1/2+\beta+it,\chi_D)|^k} \leq q^{2g+Ag/(\log g)^{1/4}} \exp \Big( k^2 (\log g)^2 (\log \log g) \Big).
\end{equation}
(Recall that $A$ could change from line to line and does not depend on $k$). 

Now using H\"{o}lder's inequality, we have that 
\begin{align*}
\sum_{D \in \mathcal{H}_{2g+1}}   \frac{1}{|L(1/2+\beta+it,\chi_D)|^k} \leq  \Big( \sum_{D \in \mathcal{H}_{2g+1}} \frac{1}{|L(1/2+\beta+it,\chi_D)|^{km}} \Big)^{\frac{1}{m}} \Big(\sum_{D \in \mathcal{H}_{2g+1}} 1 \Big)^{\frac{m-1}{m}},
\end{align*}
and using the bound \eqref{to_improve}, we get that 
\begin{align*}
\sum_{D \in \mathcal{H}_{2g+1}}   \frac{1}{|L(1/2+\beta+it,\chi_D)|^k} \leq q^{2g+\frac{Ag}{m(\log g)^{1/4}}} \exp \Big(m k^2 (\log g)^2 \log \log g \Big).
\end{align*}
We choose $m= \frac{\sqrt{g}}{k (\log g)^{9/8}\sqrt{\log \log g}},$ and then it follows that
\begin{align}
\sum_{D \in \mathcal{H}_{2g+1}}   \frac{1}{|L(1/2+\beta+it,\chi_D)|^k} \leq q^{2g+C k \sqrt{g}(\log g)^{7/8}\sqrt{\log \log g}},
\label{to_improve2}
\end{align}
for some $C>0$ which does not depend on $k$.
\subsection{The case $\beta \gg g^{-\frac{1}{k}+\epsilon}$, the second step}
We will now repeat the argument above to improve the bound \eqref{to_improve2}. Throughout the argument, rather than using the large sieve inequality (Lemmas \ref{sum_primes}, \ref{j+1}), we will use Lemmas \ref{variant1} and \ref{variant2}. We again first assume that $k \geq 1$.

We will make the following choice of parameters:
\begin{equation}
\label{n_0_again}
N_0 = \Big[\frac{ \sqrt{g}(d-1/2)(\log g)^{1/8}}{4Ck (\log q)  \sqrt{\log \log g}}\Big], \, s_0= 2 \Big[ \frac{g}{2N_0} \Big], \, \ell_0=2 \lceil s_0^d/2 \rceil,
\end{equation}
where $C$ is the constant in \eqref{to_improve}.
%\acom{$$N_0 = \Big[\frac{ \sqrt{g}(d-1/2)(\log g)^{1/8}}{4Ck \sqrt{\log \log }}\Big].$$}
For $1 \leq j \leq K$, we pick $N_j, s_j,\ell_j$ as in \eqref{n1j}. We choose $a,d,r$ such that
\begin{equation}
 \rat = 1-4k\epsilon. \label{rat5} 
 \end{equation} 
  For example, we can pick 
\begin{equation}
a=  \frac{2(1-3k\epsilon)}{1-2k\epsilon}, \, d=\frac{2-7k\epsilon}{2-6k\epsilon}, \, r=\frac{1}{1-2k\epsilon}.
\label{adr3}
\end{equation}
and we choose $K$ such that $N_K$ is the largest integer of the form \eqref{n1j} such that
%\begin{equation}
%N_K=[cg]-1,\label{nk4}
%\end{equation} for $c$ a constant such that
%\begin{equation}
%c^{1-d} \leq \frac{(2-a)(r^{1-d}-1)}{r^{1-d} a^d}.
%\label{condition_c}
%\end{equation}
\begin{equation}
\label{nk5}
N_K \leq c_1 g,
\end{equation}
where $c_1>0$ is a small constant such that 
\begin{equation}
\frac{c_1^{1-d} a^d 2^{1-d}}{r^{1-d}-1} g +  \frac{4c_1r}{r-1} g + g^d N_0^{1-d} < (2-a) g.
\label{condition_c1}
\end{equation}
Note that it is possible to choose such a constant $c_1$ since the last term above is of size $o(g)$.

%\begin{equation}
%N_K \leq  \Big[\frac{\log_q(g\beta)}{\beta}\Big]+1.
%\label{nk5}
%\end{equation}
Also note that the condition \eqref{condition_c1} above ensures that 
\begin{equation*}
\sum_{h=0}^K \ell_h N_h \leq 2g,
\end{equation*}
and that $$\sum_{h=0}^j \ell_h N_h+s_{j+1}N_{j+1} \leq 2g,$$
for $j \leq K-1$.
Indeed, note that
\begin{align*}
\sum_{h=0}^K & \ell_h N_h + s_{j+1} N_{j+1} \leq \ell_0 N_0 + \sum_{h=1}^K \ell_h N_h + a g \leq  (s_0^d+2)N_0 + \sum_{h=1}^K (s_h^d+2) N_h+ag \\
& \leq g^d N_0^{1-d} +2 \sum_{h=0}^K N_0^{1-d} +\sum_{h=0}^K a^d g^d N_h^{1-d}+ag.
\end{align*}
In the above, we used the fact that $\ell_h \leq s_h^d+2$ for $h \leq K$. Further using the fact that $N_h \leq 2r^h N_0$, and computing the geometric series above shows that with the choice \eqref{condition_c1}, indeed the two conditions in Lemmas \ref{variant2} and \ref{sm1} are satisfied.

Now we proceed as in the previous step. If $D \notin \mathcal{T}_{0}$, then there exists some $0 \leq u \leq K$ such that 
$$ 1< \Big(  \frac{ke^2}{\ell_{0}} P_{I_{0}} (D;N_{u}) \Big)^{s_{0}},$$ and we then get that 
\begin{align*} 
\sum_{D \notin \mathcal{T}_{0}} &  \frac{1}{|L(1/2+\beta+it,\chi_D)|^k} \leq \sum_{D \in \mathcal{H}_{2g+1}} \frac{1}{|L(1/2+\beta+it,\chi_D)|^k}  \Big( \frac{ke^2}{\ell_{0}} P_{I_{0}}(D;N_{u}) \Big)^{s_{0}}.
\end{align*}
Rather than using the pointwise bound for the $L$-function, we use the Cauchy-Schwarz inequality and we then have
\begin{align*}
\sum_{D \notin \mathcal{T}_{0}} &  \frac{1}{|L(1/2+\beta+it,\chi_D)|^k} \leq  \Big(\frac{ke^2}{\ell_{0}} \Big)^{s_0} \Big( \sum_{D \in \mathcal{H}_{2g+1}} \frac{1}{|L(1/2+\beta+it,\chi_D)|^{2k}} \Big)^{1/2} \\
& \times  \Big( \sum_{D \in \mathcal{H}_{2g+1}}  \Big(  P_{I_{0}}(D;N_{u}) \Big)^{2s_{0}} \Big)^{1/2}.
\end{align*}

For the first term above, we can use the bound \eqref{to_improve2}, while for the second we use Lemma \ref{variant1} and Stirling's formula, and we get that
\begin{align*}
\sum_{D \notin \mathcal{T}_{0}} &  \frac{1}{|L(1/2+\beta+it,\chi_D)|^k} \leq q^{2g+1+C k \sqrt{g} (\log g)^{7/8} \sqrt{\log \log g}} \exp ( -s_0(d-1/2) \log s_0 )\\
& \times \exp \Big(s_{0} \log \Big( 2^{1/2} e^{5/2} k B(N_{0})  \Big( \log  \frac{1}{N_{0} \beta} \Big)^{\gamma(N_{0})} \sqrt{ \log N_{0}}  \Big) \Big).
\end{align*}
%Since $\beta \gg g^{-\frac{1}{k}+\epsilon}$ and $k \geq 1$, it follows that $N_0 \beta \gg 1$, so $\gamma(N_0)=0$. 
With the choice of parameters \eqref{n_0_again}, we have
\begin{equation}
\label{s02}
\sum_{D \notin \mathcal{T}_{0}}   \frac{1}{|L(1/2+\beta+it,\chi_D)|^k} = o(q^{2g}).
\end{equation}
Now we proceed as before and we have that
\begin{align*}
\sum_{D \in \mathcal{T}_{0}}  \frac{1}{|L(1/2+\beta+it,\chi_D)|^k}  \leq \sum_{D \in \mathcal{T}_{0}} S_{1}(D) + \sum_{D \in \mathcal{T}_{0}} S_{2}(D) .
\end{align*}
Using Lemmas \ref{lem_initial} and \ref{sm1}, we have that
\begin{align*}
\sum_{D \in \mathcal{T}_{0}} S_{1}(D)  & \leq q^{2g+1}  \exp \Big( O\Big(k^2 \Big( \log \frac{1}{N_K \beta} \Big)^{2 \gamma(N_K)} \Big)  \exp\Big( \frac{2gk}{N_K+1} \log \frac{1}{1-q^{-(N_K+1)\beta}} \Big)  (\log g)^{k/2}\\
& \times  N_{K}^{\frac{k}{2}+\frac{k^2B(N_{K})^2}{2} \Big( \log \frac{1}{N_{K} \beta}\Big)^{2\gamma(N_K)}}.
\end{align*}
With the choice \eqref{nk5} for $N_K$, note that we have $N_K \beta \to \infty$, so $\gamma(N_K)=0$. We have
$$ \exp\Big( \frac{2gk}{N_K+1} \log \frac{1}{1-q^{-(N_K+1)\beta}} \Big)  = O(1),$$
and then 

%\acom{Keeping in mind the choice of $N_K$, we have
%\begin{align*}
%\sum_{D \in \mathcal{T}_{0}} S_{1}(D) \leq q^{2g}   \exp \Big( O(k^2) \Big) (\log g)^{k(k+2)/2} \Big( \frac{1}{\beta} \Big)^{k(k+1)/%2}.
%\end{align*}}
%Keeping in mind the choice \eqref{nk4} for $N_K$, we have $N_{K} \beta \gg g^{\epsilon}$ (because $\beta \gg g^{-\frac{1}{k}+\epsilon}$ and $k \geq 1$), so $\gamma(N_{K})=0$. Using the expression for $B(N_K)$ (equation \eqref{b*}), it follows that
%\begin{align*}
%\sum_{D \in \mathcal{T}_{0}} S_{1}(D) \leq q^{2g}   \exp \Big( O(k^2) \Big) \exp\Big( \frac{2k}{c} \log \frac{1}{1-q^{-cg\beta}} \Big)  (\log g)^{k/2} N_{K}^{\frac{k}{2}+\frac{k^2B(N_{K})^2}{2} \Big( \log \frac{1}{N_{K} \beta} \Big)^{2\gamma(N_{K})} }.
%\end{align*}

%Note that in this case $N_{K} \beta \gg g^{\epsilon}$ (because $\beta \gg g^{-\frac{1}{k}+\epsilon}$ and $k \geq 1$), so $\gamma(N_{K})=0$ and using the expression for $N_{K}$  and for $B(N_K)$ (equations \eqref{nk4} and \eqref{b*}), it follows that 
\begin{equation}
\label{s11d2}
\sum_{D \in \mathcal{T}_{0}} S_{1}(D) \leq  q^{2g+1}  \exp \Big( O(k^2) \Big)(\log g)^{k/2}  g^{k(k+1)/2}.
\end{equation}
Now we proceed as before to deal with the term $S_2(D)$, but use Lemma \ref{variant2} instead of Lemma \ref{j+1}. Similarly to the bound \eqref{j+1_bd}, it follows that
\begin{align}
\sum_{D \in \mathcal{T}_{0}} &S_{2}(D) \leq q^{2g+1}  (\log g)^{k/2}   \sum_{j=0}^{K-1} (K-j) \exp \Big(  \frac{g \log g}{N_{j}} \Big(2k \alpha- \rat\Big) \label{to_bd5}  \\
&+ \frac{g \log N_{j}}{N_{j}} \Big( \rat-2k \Big) + \frac{ag}{r N_{j}} \log \Big( e^{3/2} c k B(N_{j+1}) \Big( \log \frac{1}{ N_{j+1}\beta} \Big)^{\gamma(N_{j+1})}\Big)\nonumber  \\
& + (\log N_{j}) \Big(B(N_{j})^2k^2(\log \frac{1}{N_{j} \beta} )^{2\gamma(N_{j};\beta)} /2+k/2 \Big) \Big) \exp \Big( O\Big(k^2 \Big( \log \frac{1}{N_j \beta} \Big)^{2 \gamma(N_j)} \Big),\nonumber 
\end{align}
for some constant $c>0$.
Because $\beta \ll g^{-\frac{1}{2k}+\epsilon}$, we have $2k\alpha-\rat \geq k\epsilon,$ in light of \eqref{rat5}. We also have that $\rat-2k<0$, since $k \geq 1$.
Now note the function in the sum over $j$ above is decreasing as a function of $N_j$ for $N_j \leq  \delta  g^{(2k\alpha-\rat)/(2k-\rat)}$ and is increasing for $N_j > \delta g^{(2k\alpha-\rat)/(2k-\rat)}$ for some constant $\delta>0$.  If $2k(2\alpha-1) \leq \rat$, then since $N_0 \asymp \sqrt{g} (\log g)^{1/8}/\sqrt{\log \log g}$, we have that $N_j> \delta g^{ (2k\alpha-\rat)/(2k-\rat)}$ for all $j \leq K-1$, and hence the function of $N_j$ in \eqref{to_bd5} is increasing.% Hence
%\begin{align*}
%\sum_{D \in \mathcal{T}_{0}} &S_{2}(D) \leq q^{2g}  K (\log g)^{k/2} \exp \Big( \frac{2kr(\alpha-1) \log g}{c} +O(1) + (\log g )\Big( \frac{k^2}{2}+\frac{k}{2} \Big) \Big)  \exp \Big( O(k^2) \Big).
%\end{align*}
%Using the fact that $\alpha -1 \leq -\epsilon$, it follows that 
%\begin{equation}
%\label{good}
% \sum_{D \in \mathcal{T}_{0}} S_{2}(D) \leq  q^{2g} g^{k(k+1)/2}.
% \end{equation}
%\acom{
Hence
\begin{align*}
\sum_{D \in \mathcal{T}_{0}} &S_{2}(D) \leq q^{2g+1}  K (\log g)^{k/2}  \exp \Big(\frac{g  \log g (\alpha-1)}{N_{K-1}} +O(1)+ \frac{k(k+1)}{2} \log g \Big).
\end{align*}
Using the fact that $\alpha-1 \leq -\epsilon$ (recall that $\beta \gg g^{-\frac{1}{k}+\epsilon}$ and $k \geq 1$), that $N_{K-1} \asymp g$ and that $K \asymp \log g$, it follows that
\begin{equation}
\label{good}
 \sum_{D \in \mathcal{T}_{0}} S_{2}(D) =o \Big( q^{2g}  g^{k(k+1)/2} \Big).
 \end{equation}

 Now if $k \geq 3/2$, then since $\beta \gg g^{-\frac{1}{k}+\epsilon}$, it follows that $2k(2\alpha-1) \leq \rat$, so the bound \eqref{good} holds in this case.

Putting things together (equations \eqref{s02}, \eqref{s11d2}, \eqref{good}) it follows that for $k \geq 3/2$ and $\beta \gg g^{-\frac{1}{k}+\epsilon}$, then 
\begin{equation}
\sum_{D \in \mathcal{H}_{2g+1}}   \frac{1}{|L(1/2+\beta+it,\chi_D)|^k} \leq  q^{2g+1}  \exp \Big( O(k^2) \Big) (\log g)^{k(k+2)/2} g^{k(k+1)/2}.
\label{3/2}
\end{equation}
This proves the bound \eqref{improved} in Theorem \ref{theorem_ub}.

If $k <3/2$, let $m$ be such that $m k \geq 3/2$. Using H\"{o}lder's inequality, we have that
\begin{align*}
\sum_{D \in \mathcal{H}_{2g+1}} &  \frac{1}{|L(1/2+\beta+it,\chi_D)|^k} \leq \Big( \sum_{D \in \mathcal{H}_{2g+1}}  \frac{1}{|L(1/2+\beta+it,\chi_D)|^{mk}} \Big)^{\frac{1}{m}} \Big( \sum_{D \in \mathcal{H}_{2g+1}} 1 \Big)^{\frac{m-1}{m}}.
\end{align*}
Since $m k \geq 3/2$, we use \eqref{3/2}, and it follows that
$$\sum_{D \in \mathcal{H}_{2g+1}}  \frac{1}{|L(1/2+\beta+it,\chi_D)|^k} \ll q^{2g} (\log g)^{k(mk+2)/2}  g^{\frac{k}{2}(mk+1)}.$$
Picking $m=3/(2k)$, we get that
$$ \sum_{D \in \mathcal{H}_{2g+1}}  \frac{1}{|L(1/2+\beta+it,\chi_D)|^k} \ll q^{2g} (\log g)^{7k/4} g^{\frac{5k}{4}}.$$
This finishes the proof of the bound \eqref{second_bd} in Theorem \ref{theorem_ub}.
\subsection{The range $\beta \ll g^{-\frac{1}{k}+\epsilon}$}

Here, we choose the parameters as follows:
\begin{equation}
\label{solution}
N_0= \Big[ \frac{\log_q g ( 4(d-1/2)+2k\alpha - \rat)}{k \alpha(1+\epsilon)} \Big], \, s_0 = 2[2g/N_0],\, \ell_0 =2 \lceil s_0^d/2 \rceil,
\end{equation}
where we choose 
\begin{equation*}
a=4-k\epsilon, \, d = \frac{8-3k\epsilon}{8-2k\epsilon}, \, r = \frac{2-k\epsilon}{2-3k\epsilon},
\end{equation*}
so that 
\begin{equation}
\rat = 2-3k\epsilon.
\label{rat2}
\end{equation}
For $1 \leq j \leq K$, we choose $N_j, s_j$ and $\ell_j$ as in \eqref{n1j}. We choose $N_K$ to be the greatest integer of the form given in \eqref{n1j} such that 
\begin{equation}
\label{nk6}
N_K \leq c_2 g,
\end{equation}
where $c_2$ is a small constant such that 
\begin{align*}
\frac{4 c_2^{1-d} a^dr^{1-d}}{r^{1-d}-1} g + \frac{8r c_2}{r-1} g + 2 (4^dg^d) N_0^{1-d} \leq (4-a) g.
\end{align*}
The third term above is of size $o(g)$, so it possible to pick such a $c_2$. Note that the above ensures that the conditions in Lemmas \ref{sum_primes}, \ref{j+1}, \ref{sm1} are satisfied (see the explanation following the choice \eqref{condition_c1}.)
We now proceed similarly as in the previous case.

If $D \notin \mathcal{T}_{0}$, then for some $0 \leq u \leq K$, we have 
\begin{align}
\sum_{D \notin \mathcal{T}_{0}} & \frac{1}{|L(1/2+\beta+it,\chi_D)|^k} \leq \sum_{D \in \mathcal{H}_{2g+1}}  \frac{1}{|L(1/2+\beta+it,\chi_D)|^k} \Big( \frac{ke^2}{\ell_{0}} \Big)^{s_{0}} (P_{I_{0}}(D;N_{u}))^{s_{0}}.\label{tb11}
\end{align}
Now using the pointwise bound in Lemma \ref{pointwise} for the $L$--function in \eqref{tb11} and Lemma \ref{sum_primes}, similarly to equation \eqref{first_contrib}, it follows that
\begin{align*}
\sum_{D \notin \mathcal{T}_{0}} & \frac{1}{|L(1/2+\beta+it,\chi_D)|^k} \leq q^{2g+g k \alpha (1+\epsilon)+Ag/(\log g)^{1/4}} \exp \Big(-s_{0} (d-1/2) \log s_{0}  \Big)\\
& \times \exp \Big(s_{0} \log \Big( 2^{1/2} e^{5/2} k  B(N_{0})  \Big( \log  \frac{1}{N_{0} \beta} \Big)^{\gamma(N_{0})} \sqrt{ \log N_{0}}  \Big) \Big).
\end{align*}
 Note that with the choice \eqref{solution}, it follows that (after a relabeling of the $\epsilon$):
 \begin{equation}
\sum_{D \notin \mathcal{T}_{0}}   \frac{1}{|L(1/2+\beta+it,\chi_D)|^k} \ll q^{g (k \alpha+1+2k \epsilon) }.\label{t101}
 \end{equation}

Now we consider the contribution from $D \in \mathcal{T}_{0}$. 
Using Lemma \ref{lem_initial}, it follows that
\begin{align*}
\sum_{D \in \mathcal{T}_{0}}  \frac{1}{|L(1/2+\beta+it,\chi_D)|^k}  \ll \sum_{D \in \mathcal{T}_{0}} S_{1}(D) + \sum_{D \in \mathcal{T}_{0}} S_{2}(D) .
\end{align*}
%Now using Lemma \ref{sm1}, we have that
%\begin{align*}
Now using Lemma \ref{sm1} and similarly to the previous case, we have that
\begin{align*}
\sum_{D \in \mathcal{T}_{0}} S_{1}(D) \ll q^{2g} \exp \Big( O\Big(k^2 \Big( \log \frac{1}{N_K \beta} \Big)^{2 \gamma(N_K)} \Big)(\log g)^{k/2} N_{K}^{\frac{k}{2}+\frac{k^2B(N_{K})^2}{2} \Big( \log \frac{1}{N_{K} \beta} \Big)^{2\gamma(N_{K})} },
\end{align*}
and hence
\begin{equation}
\label{s11md}
\sum_{D \in \mathcal{T}_{0}} S_{1}(D) \ll q^{2g} (\log g)^{k/2} \exp \Big(k^2 \Big( \log \frac{1}{g\beta} \Big)^2 \log \log g \Big).
\end{equation}
We now evaluate the contribution from $S_2(D)$. Using Lemma \ref{j+1}, and similarly to equation \eqref{j+1_bd}, we get that 
\begin{align*}
\sum_{D \in \mathcal{T}_{0}} &S_{2}(D) \leq q^{2g+ A g/(\log g)^{1/4}}  (\log g)^{k/2}   \sum_{j=0}^{K-1} (K-j)\sqrt{s_{j+1}} \exp \Big(  \frac{g \log g}{N_{j}} \Big(2k \alpha- \rat\Big)  \\
&+ \frac{g \log N_{j}}{N_{j}} \Big( \rat-2k \Big) + \frac{ag}{r N_{j}} \log \Big( 2^{1/2} e^{3/2}c k B(N_{j+1}) \Big( \log \frac{1}{ N_{j+1}\beta} \Big)^{\gamma(N_{j+1})}\Big)\nonumber  \\
& + (\log N_{j}) \Big(3B(N_{j})^2k^2(\log \frac{1}{N_{j} \beta} )^{2\gamma(N_{j})} +k/2 \Big) \Big)  \exp \Big( O \Big( k^2 \Big( \log \frac{1}{N_j \beta} \Big)^{2 \gamma(N_j)}\Big), \nonumber 
\end{align*}
for some $c>0$.
Since $\beta \ll g^{-\frac{1}{k}+\epsilon}$ and in light of \eqref{rat2}, we have that $2k \alpha - \rat \geq k\epsilon$. It follows that the maximum of the sum over $j$ above can be attained either when $j=0$ or for $j=K-1$ (the latter is possible only if $\rat-2k<0$). Plugging in the values of $N_0$ and $N_{K-1}$ it follows that the maximum is attained at $j=0$, and keeping in mind the expression \eqref{solution} for $N_0$, we obtain that 
 $$\sum_{D \in \mathcal{H}_{2g+1}} S_2(D) \ll q^{g(k\alpha+1+2k\epsilon)}.$$
 Combining the equation above, \eqref{s11md} and \eqref{t101} finishes the proof of Theorem \ref{theorem_ub}.

\section{The asymptotic formula}
\label{section_asymp}
Here, we will prove Theorem \ref{asymptotic}. 
\begin{proof}We write 
\begin{align}
\label{dir}
 \sum_{D \in \mathcal{H}_{2g+1}}   \frac{1}{ L \Big( \frac{1}{2}+\beta,\chi_D \Big)^k } = \sum_{D \in \mathcal{H}_{2g+1}} \sum_{h \in \mathcal{M}} \frac{ \chi_D(h) \tau_{-k}(h)}{|h|^{1/2+\beta}},
\end{align}
where $\tau_{-k}$ is the generalized divisor function (i.e., the multiplicative function which is given by
$$ \tau_{-k} (P^r) = (-1)^r \binom{k}{r},$$ for $P$ a prime and $r$ a positive integer, and where $\binom{k}{r}$ is the generalized binomial coefficient.)
In equation \eqref{dir}, we truncate the sum over $h$ and we write
\begin{align*}
\sum_{D \in \mathcal{H}_{2g+1}}   \frac{1}{ L \Big( \frac{1}{2}+\beta,\chi_D \Big)^k } = \sum_{D \in \mathcal{H}_{2g+1}} \sum_{h \in \mathcal{M}_{< X}} \frac{ \chi_D(h) \tau_{-k}(h)}{|h|^{1/2+\beta}}+ \sum_{D \in \mathcal{H}_{2g+1}}   \sum_{D \in \mathcal{H}_{2g+1}} \sum_{h \in \mathcal{M}_{\geq X}} \frac{ \chi_D(h) \tau_{-k}(h)}{|h|^{1/2+\beta}},
\end{align*}
for some parameter $X$ which we will choose later. For now, we can think of $X \asymp g$. Let $M_1$ denote the first term above, and $M_2$ the second. For $M_2$, we use Perron's formula \eqref{perron} for the sum over $h$ and we have 
$$M_2 = \frac{1}{2 \pi i} \oint \sum_{D \in \mathcal{H}_{2g+1}} \frac{1}{\mathcal{L} \Big(  \frac{z}{q^{1/2+\beta}},\chi_D\Big)^k z^X(z-1)} \, dz,$$
where we are integrating over a circle of radius $|z|>1$. We can pick for example $|z|=q^{\Re(\beta)/2}$, and then using Theorem \ref{theorem_ub} it follows that 
\begin{equation*}
M_2 \ll q^{2g-X \Re(\beta)/2} \Big(\frac{1}{\beta} \Big)g^{ \frac{k}{2} \Big(1+ \max \{k,3/2\} \Big)} (\log g)^{\frac{k}{2} (2+\max \{k,3/2\})}.
%\label{m2_bound}
\end{equation*}
Since $\Re(\beta) \gg \max \{g^{-1/k+\epsilon}, \log g/g\}$, we have that $1/\Re(\beta) \ll g$, so 
\begin{equation}
M_2 \ll q^{2g-X \Re(\beta)/2} g^{1+ \frac{k}{2} \Big(1+ \max \{k,3/2\} \Big)} (\log g)^{\frac{k}{2} (2+\max \{k,3/2\})}.
\label{m2_bound}
\end{equation}
We now focus on $M_1$. We write $M_1=M(\square)+M(\neq \square),$ corresponding to whether $h$ is a square or not in the expression for $M_1$.  
When $h \neq \square$, we use Lemma \ref{notpv}, and we get that 
\begin{align*}
M(\neq \square) \ll q^g \sum_{\substack{h \in \mathcal{M}_{<X} \\ h \neq \square}} \frac{|\tau_{-k}(h)|}{|h|^{1/2+\Re(\beta)}} |h|^{\epsilon}.
\end{align*}
Trivially bounding the sum over $h$ yields
\begin{equation}
M(\neq \square) \ll q^{\frac{X}{2}+g+\epsilon g}.
\label{notsq}
\end{equation}
Now for the term $M(\square)$, we use Lemma \ref{count} and we rewrite
\begin{align}
\label{square1}
M(\square) =\frac{q^{2g+1}}{\zeta_q(2)} \sum_{h \in \mathcal{M}_{<X/2}} \frac{\tau_{-k}(h^2)}{|h|^{1+2\beta}} \prod_{P|h} \Big(1+\frac{1}{|P|} \Big)^{-1} + O (q^{ \epsilon g}).
\end{align} 
We look at the generating series of the sum over $h$, and we have
\begin{align}
\sum_{h \in \mathcal{M}}  \frac{ \tau_{-h}(h^2)}{|h|^{1+2\beta}} & \prod_{P|h} \Big(1+\frac{1}{|P|} \Big)^{-1}  u^{\deg(h)} = \prod_{P} \Big(1+ \sum_{j=1}^{\infty} \frac{ \binom{k}{2j}u^{j \deg(P)}}{|P|^{j(1+2\beta)}} \Big(1+\frac{1}{|P|} \Big)^{-1}  \Big) \nonumber \\
&= \mathcal{Z} \Big(\frac{u}{q^{1+2\beta}} \Big)^{\binom{k}{2}} \mathcal{A}(u;\beta), \label{ab}
\end{align}
where $\mathcal{A}(u;\beta)$ is given by an Euler product which converges in a wider region (for example, for $|u|<\sqrt{q}$.) Using Perron's formula \eqref{perron} in \eqref{square1}, we get that
\begin{align*}
M(\square) =\frac{q^{2g+1}}{\zeta_q(2)} \frac{1}{2 \pi i}  \oint \frac{  \mathcal{Z} \Big(\frac{u}{q^{1+2\beta}} \Big)^{\binom{k}{2}} \mathcal{A}(u;\beta) }{(1-u)u^{[X/2]+1}} \, du,
\end{align*}
where we are integrating along a small circle around the origin. We can shift the contour of integration to $|u|=q^{1/2-\epsilon}$, and encounter the poles at $u=1$ and $u=q^{2\beta}$. We evaluate the residue of the pole at $u=1$ and bound the contribution from the residue at $u=q^{2\beta}$. We get that
\begin{equation*}\label{sq_t}
M(\square) = \frac{q^{2g+1}}{\zeta_q(2)} \zeta_q(1+2\beta)^{\binom{k}{2}} \mathcal{A}(1;\beta)+ O \Big( q^{2g-X \Re(\beta)} \frac{X^{k(k-1)/2-1}}{\beta}\Big).
\end{equation*}
Since $X \asymp g$ and $1/\Re(\beta) \ll g$, we have that
\begin{equation}
\label{sq_t}
M(\square) = \frac{q^{2g+1}}{\zeta_q(2)} \zeta_q(1+2\beta)^{\binom{k}{2}} \mathcal{A}(1;\beta)+ O \Big( q^{2g-X \Re(\beta)} g^{k(k-1)/2} \Big). 
\end{equation}
Now we combine the bounds \eqref{m2_bound}, \eqref{notsq} and \eqref{sq_t}, and pick $X= 2g(1-2\epsilon)$. Then it follows that
\begin{align*}
\avg \sum_{D \in \mathcal{H}_{2g+1}} & \frac{1}{ L \Big( \frac{1}{2}+\beta,\chi_D \Big)^k } = \zeta_q(1+2\beta)^{\binom{k}{2}} A(1;\beta) +O \Big(q^{-g \Re (\beta)(1-2\epsilon)} g^{1+ \frac{k}{2} \Big(1+ \max \{k,3/2\} \Big)} \\
& \times (\log g)^{\frac{k}{2} (2+\max \{k,3/2\})} \Big).
\end{align*}
The conclusion now follows after a relabeling of the $\epsilon$.
%Now if $k \geq 1$ and $\Re(\beta) \gg g^{-\frac{1}{k}+\epsilon}$, the conclusion easily follows after a relabeling of the $\epsilon$. If $k<1$ and $\beta \geq (1+5k/4+\epsilon) \log_q g/(g(1-2\epsilon))$ then the conclusion again follows after a relabeling of the $\epsilon$.
% The conclusion follows after a relabeling of the $\epsilon$.
\end{proof}
\begin{proof}[Proof of Corollary \ref{corol}]
This easily follows from Theorem \ref{asymptotic}. Indeed, if $k \geq 1$ and $\Re(\beta) \gg g^{-\frac{1}{k}+\epsilon}$ then Theorem \ref{asymptotic} provides an asymptotic formula. If $k<1$, then we rewrite Theorem \ref{asymptotic} as
$$ \avg \sum_{D \in \mathcal{H}_{2g+1}}  \frac{1}{ L \Big( \frac{1}{2}+\beta,\chi_D \Big)^k } = \zeta_q(1+2\beta)^{\binom{k}{2}} A(1;\beta) +O \Big(q^{-g\Re \beta(1-\epsilon)} g^{1+5k/4} (\log g)^{7k/4}\Big).$$
Note that  the main term above is of size $(1/\beta)^{k(k-1)/2}$. If $g \Re(\beta)(1-\epsilon)  \log q > (1+5k/4+k(1-k)/2+\epsilon) \log g$, then indeed the expression above indeed provides an asymptotic formula. Corollary \eqref{corol} follows after a relabeling of $\epsilon$.
\end{proof}

 \bibliographystyle{amsalpha}

\bibliography{Bibliography}

\end{document}